\documentclass[11pt]{amsart}
\usepackage{graphicx}
\usepackage{amscd}
\usepackage{amsmath}
\usepackage{amsfonts}
\usepackage{amssymb}
\usepackage{setspace}
\usepackage{enumerate}         
\usepackage{color}
\usepackage{url}
\usepackage{amsthm}
\usepackage{hyperref}
\usepackage{bm}
\usepackage{xy}
\usepackage{color}
\usepackage{subfig}
\usepackage{bbm}
\usepackage{etoolbox}
\allowdisplaybreaks[4]

\usepackage{natbib}

\usepackage{geometry}
\usepackage{verbatim}

\usepackage{tikz}
\usetikzlibrary{positioning,shapes,shadows,arrows}
\tikzstyle{abstract}=[rectangle, draw=black, rounded corners, fill=blue!40, drop shadow,
        text centered, anchor=north, text=white, text width=2cm]
\tikzstyle{comment}=[rectangle, draw=black, rounded corners, fill=green, drop shadow,
        text centered, anchor=north, text=white, text width=2cm]
\tikzstyle{myarrow}=[->, >=open triangle 90, thick]
\tikzstyle{line}=[-, thick]

\theoremstyle{plain}
\newtheorem{theorem}{Theorem}[section]

\newtheorem{lemma}[theorem]{Lemma}

\newtheorem{proposition}[theorem]{Proposition}

\newtheorem{assumption}[theorem]{Assumption}

\theoremstyle{remark}
\newtheorem{remark}[theorem]{Remark}

\theoremstyle{hp}

\numberwithin{equation}{section}

\newcommand{\rsto}{]\!\kern-1.8pt ]}
\newcommand{\lsto}{[\!\kern-1.7pt [}

\numberwithin{equation}{section}

\newcommand{\FF}{\mathbb{F}}

\newcommand{\HH}{\mathbb{H}}
\newcommand{\RR}{\mathbb{R}}

\newcommand{\PP}{\mathbb{P}}
\newcommand{\NN}{\mathbb{N}}

\newcommand{\cF}{\mathcal{F}}

\newcommand{\cL}{\mathcal{L}}

\newcommand{\Ex}[2]{\mathbb{E}^{#1}\left[#2\right]}                     
\newcommand{\Excond}[3]{\mathbb{E}^{#1}\left[\left.#2\right|#3\right]}  

\renewcommand{\cite}{\citet}

\makeatletter
\patchcmd{\@maketitle}
  {\ifx\@empty\@dedicatory}
  {\ifx\@empty\@date \else {\vskip3ex \centering\footnotesize\@date\par\vskip1ex}\fi
   \ifx\@empty\@dedicatory}
  {}{}
\patchcmd{\@adminfootnotes}
  {\ifx\@empty\@date\else \@footnotetext{\@setdate}\fi}
  {}{}{}
\newcommand{\subjclassname@JEL}{}
\makeatother

\begin{document}

\title[Quantization-based Scheme for FBSDEs]{A Fully Quantization-based Scheme for FBSDEs
}

\author{Giorgia Callegaro}
\address[Giorgia Callegaro]{University of Padova, Department of Mathematics, \newline
\indent via Trieste 63, 35121 Padova, Italy}
\email[Giorgia Callegaro]{gcallega@math.unipd.it}

\author{Alessandro Gnoatto}
\address[Alessandro Gnoatto]{University of Verona, Department of Economics, \newline
\indent via Cantarane 24, 37129 Verona, Italy}
\email[Alessandro Gnoatto]{alessandro.gnoatto@univr.it}

\author{Martino Grasselli}
\address[Martino Grasselli]{University of Padova, Department of Mathematics, \newline
\indent via Trieste 63, 35121 Padova, Italy\newline
\indent and\newline
\indent Devinci Research Center, L\'{e}onard de Vinci P{\^ o}le Universitaire, \newline
\indent 92916 Paris La D\'{e}fense, France}
\email[Martino Grasselli]{grassell@math.unipd.it}

\begin{abstract}
We propose a quantization-based numerical scheme for a family of decoupled FBSDEs. We simplify the scheme for the control in \cite{ps2018} so that our approach is fully based on recursive marginal quantization and does not involve  any  Monte Carlo simulation for the computation of conditional expectations. We analyse in detail the numerical error of our scheme and we show through some examples the performance of the whole procedure, which proves to be very effective in view of financial applications. 
	
\end{abstract}

\thanks{{\em Acknowledgements.} The authors are grateful to Blanka Horvath for fruitful discussions. }

\date{\today}

\maketitle

\section{Introduction and Motivation}\label{intro}

In this paper we introduce an efficient scheme for the numerical approximation of the solution $(Y,U,V)$ of a family of   Forward-Backward Stochastic Differential Equations (FBSDEs hereafter) 
\begin{equation}
	\label{eq:fbsde}
	\left\{
	\begin{array}{rcl}
		Y _ { t } &= &y _ { 0 } + \int _ { 0 } ^ { t } b \left( Y _ { s } \right) d s + \int _ { 0 } ^ { t } \sigma \left( Y _ { s } \right)^\top d W _ { s } , \quad y _ { 0 } \in \mathbb { R } ^ { d }\\
		U_{t}&=&\xi+\int_{t}^{T} f\left(s, Y_{s}, U_{s}, V_{s}\right) d s-\int_{t}^{T} V_{s}^\top d W_{s}, \quad t \in[0, T],
	\end{array}
	\right.
\end{equation}
where $W$ is a Brownian motion, $T>0$ is a deterministic terminal time and the functions $b,\sigma, f,\xi$ satisfy some conditions specified in the sequel in order to grant that the solution of \eqref{eq:fbsde} is well defined.  FBSDEs of the form \eqref{eq:fbsde} are particularly popular in financial mathematics: in typical applications the (forward) process $Y$ describes the evolution of a financial asset, while the (backward) SDE for $U$ is related to the value of 
the portfolio that hedges the terminal payoff $\xi$ through the trading strategy $V$. BSDEs allow for the treatment of non-linear pricing problems and this originated their popularity in finance. More recently, in the aftermath of the 2007-2009 financial crisis, the valuation of financial products has been revisited in several aspects, often by means of advanced BSDEs treatments. The possibility of a default of both agents involved in the transaction and the presence of multiple sources of funding are represented at the level of valuation equations by introducing typically non-linear FBSDEs for value adjustments (xVA), see e.g. \cite{capponi18}. Value adjustments are further terms are to be added or subtracted to an idealized reference price (computed in the absence of the afore-mentioned frictions), in order to obtain the final value of the transaction. For example, the computation of Credit Value Adjustment (CVA) requires the knowledge, at each time $s\in[t,T]$ (between today, $t$, and the maturity $T$) of the future probability distribution of the contingent claim. The numerical cost of such computations becomes even more pronounced when considering the whole portfolio of claims between the bank and the counterparty.

The history of BSDEs goes back to  \cite{bismut1973} and originates from the theory of stochastic optimal control.
First existence and uniqueness results have been obtained in the seminal paper of \cite{pp1990} and have been further extended in several directions, including the presence  of jumps, see \cite{Tang1994} and reflection \cite{elkaroui1997}, \cite{cvitanic1996}, \cite{bc08}, \cite{cass09}, \cite{casselk11}. 
Applications in mathematical finance are abundant. We refer the reader to \cite{ekpq1997}, \cite{gobetlw05} and the surveys in \cite{crepeyBook14}, \cite{crepey13} for numerous references and several applications in finance, both in complete and incomplete markets. 
In view of applications, an important issue concerns the approximation of the solution of a BSDE: the most relevant contribution is based on the dynamic programming approach, introduced by  \cite{brianddm02} in a Markovian setting. In this  case, the rate of convergence for deterministic time discretization has been
studied by \cite{Zhang04}, who transformed
the problem to computing a sequence of conditional expectations. This opened the door to several approaches to attack the problem, as  significant progress has been made in computing the conditional expectations: \cite{bt2004} adopted the Malliavin calculus approach, while 
\cite{gobetlw05} proposed the linear regression method based on  the Least-Squares Monte Carlo approach in \cite{longstaffs01}.
The approach of  \cite{bp2003} and \cite{dbpp2003} was based on quantization, a technique that will be treated  in the sequel as it represents  the main source of inspiration for our work. 
Since then, the literature on BSDE flourished and attained high level of generality, including the non Markovian setting. In the  case where the terminal condition  is not necessarily Markovian,  \cite{brianl14} proposed a forward scheme based on Wiener chaos expansion, for which  conditional expectations can be efficiently 
computed through the chaos decomposition formula. 

The problem of finding numerical approximations for the solution of (coupled) FBSDEs is difficult and requires additional care.  The first relevant result is due to \cite{dmp1996} and is based on the four step scheme of \cite{mpj1994}. Later, several papers were devoted to the numerical approximation of FBSDEs with reflection, such as \cite{be08}, \cite{crepeymatoussi08} and equations of  McKean-Vlasov type, see e.g. \cite{casscd19} for a numerical method based on the Picard iteration, where the motivation comes from the theory of mean field games. 
Despite the charm and mathematical beauty of these very general frameworks, what is interesting in view of  financial  applications to pricing and hedging is the case of  decoupled FBSDEs like \eqref{eq:fbsde} (that is, when the forward SDE for $Y$ does not exhibit a dependence on $U$) in a Markovian setting. In this apparently simpler setting, many challenges still remain. First,   the curse of dimensionality, namely the problem becomes immediately untractable for dimensions greater than one. Moreover, even  in the one dimensional case, the numerical procedures described above require a lot of computations together with  Monte Carlo simulations, which leads to algorithms that are too time consuming in view of concrete applications. 
Needless to say, any improvement in the efficiency of the procedure represents an extremely useful result in terms of computational time if we have in  mind portfolios that include thousands of positions.

The aim of our study is to provide a new numerical scheme for the solution of FBSDEs that allows to improve the approximation of the solution of \eqref{eq:fbsde}. We follow the spirit of \cite{bp2003} and \cite{dbpp2003}, where Pag\`es and coauthors applied the optimal quantization technique to compute the conditional expectations. We extend their approach by considering an algorithm that is  entirely based on fast quantization: in particular, our procedure does not rely on  Monte Carlo simulation in any step of the algorithm.
\\

We now give a brief picture on quantization, that can be seen as a compression technique for random variables and stochastic processes and has been widely used in many fields, including  information theory, cluster analysis, pattern recognition.
We refer to \cite{gray98} for the history of the first fifty years of quantization and to \cite{pages15} for a more recent survey focusing on  numerical probability. 
Quantization of random vectors provides the best possible discrete approximation to the original distribution, according to a distance that is commonly measured using the squared Euclidean norm.
Many numerical procedures have been developed to obtain optimal quadratic quantizers of the Gaussian
(and even non-Gaussian) distribution in high dimension, mostly
based on stochastic optimization algorithms, see  \cite{pages15} and references therein. While theoretically sound and deeply investigated, optimal quantization
typically suffers from the numerical burden that the algorithms involve. Indeed, the procedure to be performed to obtain the optimal grids is highly time-consuming, especially in the multi-dimensional case, where stochastic algorithms are necessary. 
The recursive marginal quantization, or fast quantization, introduced in \cite{PagSagna15} represents a very useful innovation  in order to overcome the computational difficulties. Sub-optimal (stationary) quantizers of the stochastic process at fixed discretization dates (hence, of random variables) are obtained in a very fast recursive way, to the point that recursive marginal quantization has been successfully applied to many models for which a (time) discretization scheme is available, see e.g.  the non exaustive list of papers: \cite{cfg2017}, \cite{cfg2018} and \cite{Fiorin2018}, \cite{ps2018} for the multi dimensional case. We also mention \cite{mcwrkp2018} where recursive quantization has been applied outside the usual Euler scheme.  \\

Here, we propose a scheme for \eqref{eq:fbsde} that is similar  to the one in \cite{ps2018}, based on recursive quantization, with a crucial difference: in a nutshell, we introduce a new discretization scheme for the control process $V$ that we express in terms of  $U$ and $Y$ instead of $U$ and the Brownian motion $W$ (details will be provided in the sequel).  This apparently small difference leads to a simpler numerical procedure, as  there will be no need to discretize the Brownian motion increments. This reduces the computational time required to solve the FBSDE. In fact, in the approximation of the conditional expectations required in our scheme, we only need the transition probabilities of the quantized process $\widehat Y$, while  \cite{ps2018} need to additionally compute a conditional expectation involving the Brownian increments,  that they have to estimate (they use  Monte Carlo simulation). Such procedure implies an additional numerical effort which is not required in our case.
In other words, once the process $Y$ has been discretized in space via recursive marginal quantization to get $\widehat Y$, we apply our backward approximation scheme in order to get an explicit and fully quantization-based algorithm.\\

We provide two numerical experiments. The first  involves a linear BSDE, so that we can test  our approximated solution in a case where there exists a closed form for the control. Here our procedure reveals to be fast and accurate. The second example focuses on a non-linear BSDE, with unknown closed-form solution, corresponding to a pricing problem where  lending and borrowing rates may be different, as in \cite{Bergman95}. We compare our solution with the one in  \cite{ps2018} that we take as a benchmark. Results are very promising insofar we get  accurate estimates even taking a very  small number of quantizers and time discretization points (20 quantizers and 50 time steps).  \\

The rest of the paper is organised as follows: in Section \ref{sec:FBSDE} we briefly introduce the FBSDE and we recall the main  existence and uniqueness results in order for our working setting to be well-posed. In Section \ref{sec:schemeFBSDE} we illustrate our new scheme for the control $U$. Section \ref{sec:quantization} provides the essentials on recursive marginal quantization that we apply in Section \ref{sec:conditional expectations} to the computation of conditional expectations.  In Section \ref{sec:error} we study the error, while in Section \ref{sec:test} we illustrate some numerical test. Section \ref{sec:conclulsion} concludes.

\section{Forward-Backward stochastic differential equations}\label{sec:FBSDE}
We start by fixing some notations. Vectors will be column vectors and, for $x\in\RR^d$, $|x|$ denotes the Euclidean norm and $\langle x,y\rangle$ denotes the inner product. Matrices are elements of $\RR^{q\times d}$, with $|y|=\sqrt{Trace[yy^\top]}$ and $\langle x,y \rangle=Trace[xy^\top]$.
Let $\left(\Omega,\cF,\PP\right)$ be a probability space rich enough to support an $\RR^q$-valued Brownian motion $W=(W_t)_{t\in[0,T]}$. Let $\FF=(\cF_t)_{t\in[0,T]}$ be the filtration generated by $W$, assumed to satisfy the standard assumptions. We consider the following spaces:
\begin{itemize}
	\item $L^2$ is the space of all $\cF_T$-measurable $\RR^d$-valued random variables $X:\Omega\mapsto\RR^d$ such that $\left\| X\right\|^2=\Ex{}{\left|X\right|^2}<\infty$.
	\item $\HH^{2,q\times d}$ is the space of all predictable $\RR^{q\times d}$-valued processes $\phi:\Omega\times [0,T]\mapsto \RR^{q\times d}$ such that $\Ex{}{\int_0^T|\phi_t|^2dt}<\infty$.
	\item $\mathbb{S}^2$ the space of all adapted processes $\phi:\Omega\times [0,T]\mapsto \RR^{q\times d}$ such that $\Ex{}{\sup_{0\leq t\leq T}|\phi_t|^2}<\infty$.
\end{itemize}

Let $Y=\left( Y _ { t } \right)_{t\in[0,T]}$ be an $\RR^d$-valued process solving the stochastic differential equation (henceforth SDE):
\begin{align}
	\label{eq:fsde}
	Y _ { t } = y _ { 0 } + \int _ { 0 } ^ { t } b \left( Y _ { s } \right) d s + \int _ { 0 } ^ { t } \sigma \left( Y _ { s } \right)^\top d W _ { s } , \quad y _ { 0 } \in \mathbb { R } ^ { d }
\end{align}
and let us consider the following standing assumption:
\begin{assumption}\label{ass:Y}
	The vector fields $b:\mathbb{R}^d \mapsto \mathbb{R}^d$ and $\sigma:\mathbb{R}^d\mapsto \mathbb{R}^{q\times d}$ satisfy the following conditions
	\begin{align}
		|b(y)-b(z)|&\leq \cL_1|y-z|,\\
		|\sigma(y)-\sigma(z)|& \leq\cL_2|y-z|,\\
		|\sigma(y)| &\leq \cL_3(1+|y|),\quad |b(y)| \leq \cL_3(1+|y|),
	\end{align}
	for some positive constants  $\cL_1,\cL_2,\cL_3$.
\end{assumption}
It is well known that under such regularity conditions there exists a unique adapted right continuous with left limits (henceforth RCLL) strong solution $Y^{y_0}=(Y^{y_0}_t)_{t\in[0,T]}$ to \eqref{eq:fsde} which is a homogeneous Markov process. It is also well known that the solution $Y^{y_0}$ satisfies the following: for all couples $(t,y_0),(t,y_0^\prime)\in[0,T]\times \RR^d$ and $p\geq 2$ we have
\begin{align}
	\mathbb{E}\left[\sup _{0 \leq t \leq T}\left|Y^{y_0}_t-y_0\right|^{p}\right] \leq \cL_4\left(1+|y_0|^{p}\right)T,\\
	\mathbb{E}\left[\sup _{0 \leq t \leq T}\left|Y^{y_0}_t-Y^{y_0^\prime}_t\right|^{p}\right] \leq \cL_5\left(|y_0-y_0^\prime|^{p}\right),
\end{align}
where $\cL_4,\cL_5$ are positive constants. To alleviate notations we will simply write $Y$ for the solution, omitting the dependence on the initial condition $y_0$. We investigate a backward SDE with a terminal condition and a generator that depends on the state process solving the forward SDE \eqref{eq:fsde}. More precisely, we consider the backward stochastic differential equation
\begin{align}
	\label{eq:bsde}
	U_{t}=\xi+\int_{t}^{T} f\left(s, Y_{s}, U_{s}, V_{s}\right) d s-\int_{t}^{T} V_{s}^\top d W_{s}, \quad t \in[0, T],
\end{align}
where $V=(V_t)_{t\in[0,T]}$ is a process in $\HH^{2,q\times 1}$.
We will also work under the following:
\begin{assumption}\label{ass:U}
	$(i)$ The function $f :[0, T] \times \mathbb{R}^{d} \times \mathbb{R} \times \mathbb{R}^{q} \rightarrow \mathbb{R}$ is Lipschitz continuous, uniformly in $t \in [0,T]$:
	$$
	\left|f(t, y, u, v)-f\left(t, y^{\prime}, u^{\prime},  v^{\prime}\right)\right| \leq \cL_6\left(\left|y-y^{\prime}\right|+\left|u-u^{\prime}\right|+\left|v-v^{\prime}\right|\right)
	$$
	for a positive constant $\cL_6$.\\
	$(ii)$ The terminal condition $\xi$ is of the form $\xi=h(Y_T)$, for a given Borel function $h : \mathbb{R}^{d} \rightarrow \mathbb{R}$.	
\end{assumption}

The system formed by the forward SDE \eqref{eq:fsde} and the backward SDE \eqref{eq:bsde} is a decoupled forward-backward SDE. Decoupled here means that the forward SDE for $Y$ does not exhibit a dependence on $U$. The following result for FBSDE is standard, see e.g. \cite{delong} Theorem~3.1.1,  Theorem~4.1.3.

\begin{theorem}
	Under  assumptions \ref{ass:Y} and \ref{ass:U} there exists a unique solution $(Y,U,V) \in \mathbb{S}^{2}(\mathbb{R}^d) \times \mathbb{S}^{2}(\mathbb{R}) \times \mathbb{H}^{2,q\times 1}$ to the FBSDE \eqref{eq:fsde}-\eqref{eq:bsde}.
\end{theorem}

\section{A generic scheme for FBSDEs}\label{sec:schemeFBSDE} 
In the following subsections we will introduce the proposed numerical scheme to approximate the solution of the FBSDE \eqref{eq:fsde}-\eqref{eq:bsde}. To do so, we fix a time discretization: let $n\in\NN$, $\Delta=\Delta_n=\frac{T}{n}$ and set $t_k=\frac{Tk}{n}$. The scheme, given below in \eqref{our_scheme}, is defined as a backward induction and reads as follows:
\begin{equation*}
	\left\{
	\begin{array}{rcl}
		\widetilde{U}_{t_n} & = & h(\overline{Y}_{t_{n}}) \quad \textrm{and for} \quad k=0, \dots, n-1  \\
		\widetilde{U}_{t_k} & = & \Excond{}{\widetilde{U}_{t_{k+1}}}{\cF_{t_k}}+ \Delta \ f({t_{k}},\overline{Y}_{t_{k}},\Excond{}{\widetilde{U}_{t_{k+1}}}{\cF_{t_k}},\widetilde{V}_{t_{k}})\\
		\widetilde{V}_{t_k} &=& \frac{1}{\Delta}  \left[ \sigma \left( \overline{Y} _ { t_k } \right)^{\top} \right]^{-1}  \Excond{}{\widetilde{U}_{t_{k+1}}\left(\overline{Y}_{t_{k+1}}-\overline{Y}_{t_k}\right)}{\cF_{t_k}}  \\
		& & -  \left[ \sigma \left( \overline{Y} _ { t_k } \right)^{\top} \right]^{-1}  \Excond{}{\widetilde{U}_{t_{k+1}}}{\cF_{t_k}}b \left( \overline{Y} _ { t_{k} } \right) \\
	\end{array}
	\right.
\end{equation*}
where $\widetilde{U}$ and $\widetilde{V}$ are approximations of $U$ and $V$ (that will be properly introduced in Subsection \ref{subsec:truncation}) and where $\overline Y$ denotes a suitable (time) discretization of $Y$ that, at this point, is left unspecified. The scheme is similar to the one proposed in \cite{ps2018}, the novelty being a new discretization scheme for the control process.  More precisely, $\widetilde V_{t_k}, k=0, \dots, n-1,$ is no longer a function of $\widetilde U_{t_{k+1}}, W_{t_k},  W_{t_{k+1}}$, but depends here only on $\widetilde U_{t_{k+1}}, \overline Y_{t_k},  \overline Y_{t_{k+1}}$. This leads to a simpler numerical procedure, which is faster to implement. Indeed, since $Y$ is approximated indipendently from $U$ and $V$, there will be no need to discretize the Brownian motion increments and this will result in a speed-up of the computational time required to solve the FBSDE. More details on this will be given in Remark \ref{rem:diff_Pages}.
From a practical point of view, once the stochastic process $\overline Y$ has been discretized in space via recursive marginal quantization to get $\widehat Y$, the scheme reads as in Equation \eqref{our_scheme_quant} and the backward recursion results to be explicit and fully quantization-based.

\subsection{Scheme for the value process $U$}\label{sec:scheme_U} 
\textcolor{black}{Following \cite{zhao2006} we provide a step by step derivation of the numerical scheme for the process $U$.} 
Let $(Y,U,V)$ be the adapted solution to the FBSDE \eqref{eq:fsde}-\eqref{eq:bsde}. Restricting ourselves to two consecutive points in time $t_{k+1}$ and $t_k$, we write
\begin{align}\label{eq:U_discretized}
	U_{t_k}=U_{t_{k+1}}+\int_{t_k}^{t_{k+1}}f(s,Y_s,U_s,V_s)ds-\int_{t_k}^{t_{k+1}}V^\top_sdW_s
\end{align}
and taking $\cF_{t_k}$-conditional expectations on both sides we get
\begin{align*}
	U_{t_k}=\Excond{}{U_{t_{k+1}}}{\cF_{t_k}}+\int_{t_k}^{t_{k+1}}\Excond{}{f(s,Y_s,U_s,V_s)}{\cF_{t_k}}ds.
\end{align*}
Let us first concentrate on the integral term, using $\theta_1\in[0,1]$ we write
\begin{align*}
	\begin{aligned}
		\int_{t_k}^{t_{k+1}}\Excond{}{f(s,Y_s,U_s,V_s)}{\cF_{t_k}}ds&=(t_{k+1}-t_k)\left\{(1-\theta_1)\Excond{}{f({t_{k+1}},Y_{t_{k+1}},U_{t_{k+1}},V_{t_{k+1}})}{\cF_{t_k}}\right.\\
		&\left.\quad+\theta_1f({t_{k}},Y_{t_{k}},U_{t_{k}},V_{t_{k}})\right\}+R^U
	\end{aligned}
\end{align*}
where the error term $R^U$ is defined as
\begin{align*}
	\begin{aligned}
		R^U&:=\int_{t_k}^{t_{k+1}}\left(\Excond{}{f(s,Y_s,U_s,V_s)}{\cF_{t_k}}-(1-\theta_1)\Excond{}{f({t_{k+1}},Y_{t_{k+1}},U_{t_{k+1}},V_{t_{k+1}})}{\cF_{t_k}}\right.\\
		&\left.\quad\quad\quad\quad\quad\quad+\theta_1f({t_{k}},Y_{t_{k}},U_{t_{k}},V_{t_{k}})\right)ds.
	\end{aligned}
\end{align*}
Hence we arrive at
\begin{align*}
	U_{t_k}&=\Excond{}{U_{t_{k+1}}}{\cF_{t_k}}+(t_{k+1}-t_k)\left\{(1-\theta_1)\Excond{}{f({t_{k+1}},Y_{t_{k+1}},U_{t_{k+1}},V_{t_{k+1}})}{\cF_{t_k}}\right.\\
	&\left.\quad+\theta_1f({t_{k}},Y_{t_{k}},U_{t_{k}},V_{t_{k}})\right\}+R^U .
\end{align*}
\begin{remark}
	In most situations, we do not have an exact simulation scheme for the solution of the forward SDE \eqref{eq:fsde}. This means in general that we are not able to simulate $Y$ (i.e. the exact solution of \eqref{eq:fsde}), and we need to introduce a suitable discretization $\overline{Y}$, such as the Euler-Maruyama scheme, the Milstein discretization or higher order scheme as presented e.g. in \cite{kloedenPlaten}.
\end{remark}
For the moment, let $\overline{Y}$ be a discretization scheme for $Y$, which is still left unspecified. We write
\begin{align*}
	\begin{aligned}
		U_{t_k}&=\Excond{}{U_{t_{k+1}}}{\cF_{t_k}}+(t_{k+1}-t_k)\left\{(1-\theta_1)\left(\Excond{}{f({t_{k+1}},\overline{Y}_{t_{k+1}},U_{t_{k+1}},V_{t_{k+1}})}{\cF_{t_k}}+R^{f1}\right)\right.\\
		&\left.\quad+\theta_1\left(f({t_{k}},\overline{Y}_{t_{k}},U_{t_{k}},V_{t_{k}})+R^{f2}\right)\right\}+R^U, 
	\end{aligned}
\end{align*}
where
\begin{align*}
	R^{f1}&:= \Excond{}{f({t_{k+1}},Y_{t_{k+1}},U_{t_{k+1}},V_{t_{k+1}})}{\cF_{t_k}}-\Excond{}{f({t_{k+1}},\overline{Y}_{t_{k+1}},U_{t_{k+1}},V_{t_{k+1}})}{\cF_{t_k}}\\
	R^{f2}&:=f({t_{k}},Y_{t_{k}},U_{t_{k}},V_{t_{k}})-f({t_{k}},\overline{Y}_{t_{k}},U_{t_{k}},V_{t_{k}}).
\end{align*}
Setting $R^{f}:=(1-\theta_1)R^{f1}+\theta_1R^{f2}$ we finally arrive at
\begin{align}
	\label{eq:valueProcessWithErrors}
	\begin{aligned}
		U_{t_k}&=\Excond{}{U_{t_{k+1}}}{\cF_{t_k}}+(t_{k+1}-t_k)\left\{(1-\theta_1)\Excond{}{f({t_{k+1}},\overline{Y}_{t_{k+1}},U_{t_{k+1}},V_{t_{k+1}})}{\cF_{t_k}}\right.\\
		&\left.\quad+\theta_1f({t_{k}},\overline{Y}_{t_{k}},U_{t_{k}},V_{t_{k}})\right\}+R^U+R^{f}.
	\end{aligned}
\end{align}
We observe that in \eqref{eq:valueProcessWithErrors} the discretization error is due to the time discretization and the choice of the numerical scheme for the forward process $Y$. Further sources of error will arise in the space dimension as we will approximate the conditional expectations appearing in \eqref{eq:valueProcessWithErrors}.

\subsection{Scheme for the control}
We derive the newly proposed scheme for the numerical approximation of the control process $V$. What is tipically done in the literature is obtaining a discretization scheme for $V$ which involves the increments of the Brownian motion. This is done by multiplying Equation \eqref{eq:U_discretized} by $(W_{t_{k+1}} - W_{t_k})$ and then taking as usual conditional expectations and truncating the error terms.
We will proceed here in a different way, which is new, up to our knowledge. Our objective, indeed, is to derive an update rule for the control that only involves $Y$ (i.e. the process that we will quantize in the sequel) and not $W$. 
To this end we consider again the BSDE \eqref{eq:U_discretized} and multiply both sides by $\int_{t_k}^{t_{k+1}}\sigma(Y_s)^\top dW_s$:
\begin{align}
	\begin{aligned}\label{eq:scheme_U_0}
		U_{t_k}\int_{t_k}^{t_{k+1}}\sigma(Y_s)^\top dW_s&=U_{t_{k+1}}\int_{t_k}^{t_{k+1}}\sigma(Y_s)^\top dW_s+\int_{t_k}^{t_{k+1}}f(s,Y_s,U_s,V_s)ds\int_{t_k}^{t_{k+1}}\sigma(Y_s)^\top dW_s\\
		&\quad-\int_{t_k}^{t_{k+1}}V^\top_sdW_s\int_{t_k}^{t_{k+1}}\sigma(Y_s)^\top dW_s.
	\end{aligned}
\end{align}
We take then $\cF_{t_k}$-conditional expectations on both sides, thus obtaining the following identity
\begin{align}
	\label{eq:bsdeMultiplied}
	\begin{aligned}
		\underbrace{U_{t_k}\Excond{}{\int_{t_k}^{t_{k+1}}\sigma(Y_s)^\top dW_s}{\cF_{t_k}}}_{(A)}&=\underbrace{\Excond{}{U_{t_{k+1}}\int_{t_k}^{t_{k+1}}\sigma(Y_s)^\top dW_s}{\cF_{t_k}}}_{(B)}\\
		&+ \underbrace{\Excond{}{\int_{t_k}^{t_{k+1}}f(s,Y_s,U_s,V_s)ds\int_{t_k}^{t_{k+1}}\sigma(Y_s)^\top dW_s}{\cF_{t_k}}}_{(C)}\\
		&- \underbrace{ \Excond{}{\int_{t_k}^{t_{k+1}}V^\top_sdW_s\int_{t_k}^{t_{k+1}}\sigma(Y_s)^\top dW_s}{\cF_{t_k}}}_{(D)}.
	\end{aligned}
\end{align}
We now analyze every conditional expectation in \eqref{eq:bsdeMultiplied} starting from $(D)$:
\begin{itemize}
	\item{(D)} Via It\^o isometry we find
$		\Excond{}{\int_{t_k}^{t_{k+1}}V^\top_sdW_s\int _ { t_k } ^ { t_{k+1} } \sigma \left( Y _ { s } \right)^\top d W _ { s }}{\cF_{t_k}}=\Excond{}{\int_{t_k}^{t_{k+1}}\sigma \left( Y _ { s } \right)^\top V_s ds}{\cF_{t_k}} $
	and using $\theta_2\in[0,1]$ we have
	\begin{align*}
		\begin{aligned}
			\Excond{}{\int_{t_k}^{t_{k+1}}\sigma \left( Y _ { s } \right)^\top V_s ds}{\cF_{t_k}}&=(t_{k+1}-t_k)\left\{(1-\theta_2)\Excond{}{\sigma \left( Y _ { t_{k+1} } \right)^\top V_{t_{k+1}}}{\cF_{t_k}}+\theta_2\sigma \left( Y _ { t_k } \right)^\top V_{t_k}\right\}\\
			&\quad+R^{V-\theta},
		\end{aligned}
	\end{align*}
	where
	$	R^{V-\theta}:=\Excond{}{\int_{t_k}^{t_{k+1}} \left[\sigma \left( Y _ { s } \right)^\top V_s-(1-\theta_2)\sigma \left( Y _ { t_{k+1} } \right)^\top V_{t_{k+1}}-\theta_2\sigma \left( Y _ { t_k } \right)^\top V_{t_k} \right] ds}{\cF_{t_k}}$.
	We now take into account the impact of the numerical scheme to approximate $Y$, namely we insert $\overline{Y}$:
	\begin{align*}
		\begin{aligned}
			\Excond{}{\int_{t_k}^{t_{k+1}}\sigma \left( Y _ { s } \right)^\top V_s ds}{\cF_{t_k}}&=(t_{k+1}-t_k)\left\{(1-\theta_2)\Excond{}{\sigma \left( \overline{Y} _ { t_{k+1} } \right)^\top V_{t_{k+1}}}{\cF_{t_k}}+\theta_2\sigma \left( \overline{Y} _ { t_k } \right)^\top V_{t_k}\right\}\\
			&+R^{V-\theta}+R^{V-Y},
		\end{aligned}
	\end{align*}
	with
	$	R^{V-Y}:=(t_{k+1}-t_k)\left\{(1-\theta_2)\left(\sigma \left({Y} _ { t_{k+1} }\right)^\top -\sigma \left( \overline{Y} _ { t_{k+1} } \right)^\top\right)V_{t_{k+1}}+\theta_2\left(\sigma \left({Y} _ { t_{k+1} }\right)^\top -\sigma \left( \overline{Y} _ { t_{k+1} } \right)^\top\right)V_{t_k}\right\} $.
	\item{(C)} We clearly have:
	\begin{align*}
		\Excond{}{\int_{t_k}^{t_{k+1}}f(s,Y_s,U_s,V_s)ds\int _ { t_k } ^ { t_{k+1} } \sigma \left( Y _ { s } \right)^\top d W _ { s }}{\cF_{t_k}}=0.
	\end{align*}
	\item{(A)} Here, too:
	\begin{align*}
		U_{t_{k}}\Excond{}{\int _ { t_k } ^ { t_{k+1} } \sigma \left( Y _ { s } \right)^\top d W _ { s }}{\cF_{t_k}}&=0.
	\end{align*}
	\item{(B)} A distinctive feature of our numerical scheme is based on the following simple observation: we can exploit the dynamics \eqref{eq:fsde} to express the stochastic integral in $(B)$  as follows
	\begin{align}
		\label{eq:crucial}
		\Excond{}{U_{t_{k+1}}\int _ { t_k } ^ { t_{k+1} } \sigma \left( Y _ { s } \right)^\top d W _ { s }}{\cF_{t_k}}&=\Excond{}{U_{t_{k+1}}\left(Y_{t_{k+1}}-Y_{t_k}-\int _ { t_k } ^ { t_{k+1} } b \left( Y _ { s } \right) ds\right)}{\cF_{t_k}}.
	\end{align}
	Splitting the conditional expectation on the right hand side, we obtain two simple conditional expectations that can be suitably estimated, once we have an approximation for the transition probabilities of the forward process $Y$.  We write
	\begin{align*}
		\Excond{}{U_{t_{k+1}}\left(Y_{t_{k+1}}-Y_{t_k}\right)}{\cF_{t_k}}=\Excond{}{U_{t_{k+1}}\left(\overline{Y}_{t_{k+1}}-\overline{Y}_{t_k}\right)}{\cF_{t_k}}+R^{U-Y},
	\end{align*}
	where
	$	R^{U-Y}:=\Excond{}{U_{t_{k+1}}\left(Y_{t_{k+1}}-\overline{Y}_{t_{k+1}}\right)}{\cF_{t_k}}-\Excond{}{U_{t_{k+1}}}{\cF_{t_k}}\left(Y_{t_{k}}-\overline{Y}_{t_{k}}\right)$,
	while for the second expectation in \eqref{eq:crucial} we have
	\begin{align*}
		\begin{aligned}
			\Excond{}{U_{t_{k+1}}\int _ { t_k } ^ { t_{k+1} } b \left( Y _ { s } \right) ds}{\cF_{t_k}}&=(t_{k+1}-t_k)(1-\theta_2)\Excond{}{U_{t_{k+1}}b \left( \overline{Y} _ { t_{k+1} } \right)}{\cF_{t_k}}\\
			& \ +(t_{k+1}-t_k)\theta_2\Excond{}{U_{t_{k+1}}}{\cF_{t_k}}b \left( \overline{Y} _ { t_{k} } \right)\\
			& \ +R^{b-\theta}+R^{b-Y},
		\end{aligned}
	\end{align*}
	where 
	\begin{align*}
		\begin{aligned}
			R^{b-\theta}&:=\Excond{}{U_{t_{k+1}}\left(\int _ { t_k } ^ { t_{k+1} } b \left( Y _ { s } \right) ds-(t_{k+1}-t_k)\left\{(1-\theta_2)b \left( Y _ { t_{k+1} } \right)+\theta_2b \left( Y _ { t_{k} } \right)\right\}\right)}{\cF_{t_k}}\\
			R^{b-Y}&:=(t_{k+1}-t_k)(1-\theta_2)\Excond{}{U_{t_{k+1}}\left(b \left( Y _ { t_{k+1} } \right)-b \left( \overline{Y} _ { t_{k+1} } \right)\right)}{\cF_{t_k}}\\
			& \ +(t_{k+1}-t_k)\theta_2\Excond{}{U_{t_{k+1}}}{\cF_{t_k}}\left(b \left( Y _ { t_{k} } \right)-b \left( \overline{Y} _ { t_{k} } \right)\right).
		\end{aligned}
	\end{align*}
\end{itemize}
By regrouping all terms $(A), (B), (C)$ and $(D)$ we obtain the following relation, providing an implicit update rule for the control process $V$ (the explicit rule for the control $V$ will be specified in the next subsection): 
\begin{align}
	\label{eq:updateControl}
	\begin{aligned}
		0&=\Excond{}{U_{t_{k+1}}\left(\overline{Y}_{t_{k+1}}-\overline{Y}_{t_k}\right)}{\cF_{t_k}}+R^{U-Y}  -(t_{k+1}-t_k)(1-\theta_2)\Excond{}{U_{t_{k+1}}b \left( \overline{Y} _ { t_{k+1} } \right)}{\cF_{t_k}}\\
		& \ -(t_{k+1}-t_k)\theta_2\Excond{}{U_{t_{k+1}}}{\cF_{t_k}}b \left( \overline{Y} _ { t_{k} } \right)  -R^{b-\theta}-R^{b-Y}\\
		& \ -(t_{k+1}-t_k)\left\{(1-\theta_2)\Excond{}{\sigma \left( \overline{Y} _ { t_{k+1} } \right)^\top V_{t_{k+1}}}{\cF_{t_k}} + \theta_2 \ \sigma \left( \overline{Y} _ { t_k } \right)^\top V_{t_k}\right\}  -R^{V-\theta}-R^{V-Y}.
	\end{aligned}
\end{align}

\subsection{The truncated scheme}\label{subsec:truncation}
Starting from Equations \eqref{eq:valueProcessWithErrors} and \eqref{eq:updateControl} and by truncating all error terms, we obtain the following system of two equations (for each $k$) for the couple $(\widetilde{U},\widetilde{V})$, where  $(\widetilde{U},\widetilde{V})$ are approximations of  $({U},{V})$ where we recall that $\overline Y$ is a suitable discretization of $Y$
\begin{align}
	\begin{aligned}
		\widetilde{U}_{t_k}&=\Excond{}{\widetilde{U}_{t_{k+1}}}{\cF_{t_k}}+(t_{k+1}-t_k)\left\{(1-\theta_1)\Excond{}{f({t_{k+1}},\overline{Y}_{t_{k+1}},\widetilde{U}_{t_{k+1}},\widetilde{V}_{t_{k+1}})}{\cF_{t_k}}\right.\\
		&\left.\quad+\theta_1f({t_{k}},\overline{Y}_{t_{k}},\widetilde{U}_{t_{k}},\widetilde{V}_{t_{k}})\right\}
	\end{aligned}\\
	\begin{aligned}
		0&=\Excond{}{\widetilde{U}_{t_{k+1}}\left(\overline{Y}_{t_{k+1}}-\overline{Y}_{t_k}\right)}{\cF_{t_k}}\\
		& \ -(t_{k+1}-t_k)(1-\theta_2)\Excond{}{\widetilde{U}_{t_{k+1}}b \left( \overline{Y} _ { t_{k+1} } \right)}{\cF_{t_k}}\\
		& \ -(t_{k+1}-t_k)\theta_2\Excond{}{\widetilde{U}_{t_{k+1}}}{\cF_{t_k}}b \left( \overline{Y} _ { t_{k} } \right)\\
		& \ -(t_{k+1}-t_k)\left\{(1-\theta_2)\Excond{}{\sigma \left( \overline{Y} _ { t_{k+1} } \right)^\top \widetilde{V}_{t_{k+1}}}{\cF_{t_k}} + \theta_2\sigma \left( \overline{Y} _ { t_k } \right)^\top \widetilde{V}_{t_k}\right\}.
	\end{aligned}
\end{align}

\begin{remark}
	The second equation above (which is the truncation of Equation \eqref{eq:updateControl}) provides an approximation scheme for $\widetilde V_{t_k}$ as a function of $\widetilde V_{t_{k+1}}, \widetilde U_{t_{k+1}}, \overline Y_{t_k}, \overline Y_{t_{k+1}}$.
\end{remark}

In particular, if we set $\theta_1=\theta_2=1$, we obtain the recursive scheme (which is not yet fully explicit):

\begin{equation*}
	\begin{cases}
		\widetilde{U}_{t_k}=\Excond{}{\widetilde{U}_{t_{k+1}}}{\cF_{t_k}}+(t_{k+1}-t_k)f({t_{k}},\overline{Y}_{t_{k}},\widetilde{U}_{t_{k}},\widetilde{V}_{t_{k}})\\
		\widetilde{V}_{t_k} = \frac{1}{(t_{k+1}-t_k)}  \left[ \sigma \left( \overline{Y} _ { t_k } \right)^{\top} \right]^{-1}  \Excond{}{\widetilde{U}_{t_{k+1}}\left(\overline{Y}_{t_{k+1}}-\overline{Y}_{t_k}\right)}{\cF_{t_k}} \\
		\quad\quad \quad\quad \quad\quad-   \left[ \sigma \left( \overline{Y} _ { t_k } \right)^{\top} \right]^{-1}  \Excond{}{\widetilde{U}_{t_{k+1}}}{\cF_{t_k}}b \left( \overline{Y} _ { t_{k} } \right),
	\end{cases}
\end{equation*}
where $ \left[ \sigma \left( \overline{Y} _ { t_k } \right)^{\top} \right]^{-1} $ denotes the $(q \times d)$ left-inverse of the matrix $\sigma \left( \overline{Y} _ { t_k } \right)^{\top} $. 
\begin{remark}
	In Section \ref{sec:error}, focusing on the error analysis, we will for simplicity consider the case when $q=d$ and we will work under Assumption \ref{ellipticity}, which will guarantee the invertibility of $\sigma$.
\end{remark}

\begin{remark}
	In \cite{ps2018} the scheme is made fully explicit by performing a conditioning inside the driver, which results in the following
	\begin{equation}
		\label{scheme_ps}
		\left\{
		\begin{array}{rcl}
			\widetilde{U}_{t_n} & = & h(\overline{Y}_{t_{n}}) \quad \textrm{and for} \quad k=0, \dots, n-1 \\
			\widetilde{U}_{t_k} & = & \Excond{}{\widetilde{U}_{t_{k+1}}}{\cF_{t_k}}+(t_{k+1}-t_k)f({t_{k}},\overline{Y}_{t_{k}},\Excond{}{\widetilde{U}_{t_{k+1}}}{\cF_{t_k}},\widetilde{V}_{t_{k}})\\
			\widetilde{V}_{t_k}^{\textrm{PS}} &=&\frac{1}{(t_{k+1} - t_k)} \Excond{}{\widetilde{U}_{t_{k+1}}\left( W_{t_{k+1}} - W_{t_k}\right)}{\cF_{t_k}} .\\
		\end{array}
		\right.
	\end{equation}
\end{remark}
So, borrowing this idea, we are now in a position to finally state our proposed scheme as (recall that $(t_{k+1} - t_k) =\Delta$ for every $k=0, \dots, n-1$):
\begin{equation}
	\label{our_scheme}
	\left\{
	\begin{array}{rcl}
		\widetilde{U}_{t_n} & = & h(\overline{Y}_{t_{n}}) \quad \textrm{and for} \quad k=0, \dots, n-1  \\
		\widetilde{U}_{t_k} & = & \Excond{}{\widetilde{U}_{t_{k+1}}}{\cF_{t_k}}+ \Delta \ f({t_{k}},\overline{Y}_{t_{k}},\Excond{}{\widetilde{U}_{t_{k+1}}}{\cF_{t_k}},\widetilde{V}_{t_{k}})\\
		\widetilde{V}_{t_k} &=& \frac{1}{\Delta}  \left[ \sigma \left( \overline{Y} _ { t_k } \right)^{\top} \right]^{-1}  \Excond{}{\widetilde{U}_{t_{k+1}}\left(\overline{Y}_{t_{k+1}}-\overline{Y}_{t_k}\right)}{\cF_{t_k}}  \\
		& & -  \left[ \sigma \left( \overline{Y} _ { t_k } \right)^{\top} \right]^{-1}  \Excond{}{\widetilde{U}_{t_{k+1}}}{\cF_{t_k}}b \left( \overline{Y} _ { t_{k} } \right). \\
	\end{array}
	\right.
\end{equation}

We stress that, for the moment, we obtained a general, yet original, discretization for the FBSDE \eqref{eq:fsde}-\eqref{eq:bsde}. The role of recursive marginal quantization will become apparent as we approximate the conditional expectations appearing in the scheme above.

\section{A primer on recursive product marginal quantization}\label{sec:quantization}
We provide some background on recursive marginal quantization. We consider a diffusion process $Y$ as in \eqref{eq:fsde} and its discretized version $\overline{Y}$ over a given time grid. Quantizing the diffusion process $Y$ via recursive marginal quantization (henceforth RMQ) means the following: we consider the discretized analog $\overline{Y}$ of $Y$ and, for each given point in time, we project every single random variable $\overline{Y}_{t_{k+1}}$ on a finite grid of points by exploiting the fact that the conditional law of $\overline{Y}_{t_{k+1}}$ given its value at time $t_k$ is known. When the discretization $\overline{Y}$ is chosen to be the Euler scheme the conditional law of $\overline{Y}_{t_{k+1}}$ given $$\overline{Y}_{t_{k}}$$ is Gaussian. This technique was first introduced in \cite{PagSagna15} and was further developed in \cite{Fiorin2018} and applied in different settings, such as \cite{cfg2017} among others.

Let us now provide a minimum insight on RMQ. The Euler scheme of $Y$, solution of \eqref{eq:fsde}, is defined via the recursion
\begin{align}
	\label{eq:EulerScheme}
	\overline { Y } _ { t _ { k + 1 } } = \overline { Y } _ { t _ { k } } + \Delta \ b \left( \overline { Y } _ { t _ { k } } \right) + \sigma \left( \overline { Y } _ { t _ { k } } \right)^{\top} \left( W _ { t _ { k + 1 } } - W _ { t _ { k } } \right) , \quad \overline { y } _ { 0 } \in \mathbb { R } ^ { d },
\end{align}
for $\Delta = \Delta _ { n } = \frac { T } { n }$ and $t _ { k } = \frac { k T } { n }$. For notational simplicity, in this section we set $\overline { Y }_k:=\overline { Y }_{t_k}$.
\begin{remark} 
	 Some extensions are possible:
	\begin{itemize}
		\item[a)] The results presented here can be extended without any technical issue, yet with additional notational burden, to the case when the coefficients $b$ and $\sigma$ are no longer time homogeneous (this is the setting in \cite{ps2018}). 
		\item[b)] It is possible to consider higher order schemes such as e.g. the Milstein discretization as in \cite{mcwrkp2018}. This has an obvious implication on the shape of the conditional distribution of $\overline { Y } _ {  { k + 1 } }$ given $\overline { Y } _ {  { k } }$.
	\end{itemize}
\end{remark}
%

We define the Euler operator, which allows one to express the distribution of $\overline { Y }^{\ell} _ {  k  +1}$ given $\{ \overline { Y }^{\ell} _ {  k  }=y \}$, with $\ell =1,\dots, d$
\begin{align}\label{Euler_op}
	\mathcal { E } _ { k } ( y , z ) : = y + \Delta \ b \left( y \right) + \sqrt { \Delta } \ \sigma \left( y \right)^{\top} z , \quad y \in \mathbb { R } ^ { d } , z \in \mathbb { R } ^ { q }.
\end{align}
Our target is discretizing $\overline { Y } _ { k + 1 } $ via a finite grid, $\Gamma_{k+1}$, under the constraint that the resulting approximating error has to be minimal.
Namely, using the Euler operator, we consider the $L^2$-distortion at time $t_{k+1}$, $\overline { D } _ { k + 1 }$, which is defined as the square of the $L^2$-distance between the random variable $\overline Y_{k+1}$ and the grid $\Gamma_{k+1}$
\begin{align}
	\begin{aligned} \overline { D } _ { k + 1 } \left( \Gamma _ { k + 1 } \right) & = \mathbb { E } \left[ \left( \operatorname { dist } \left( \overline { Y } _ { k + 1 } \Gamma _ { k + 1 } \right) ^ { 2 } \right] \right. = \mathbb { E } \left[ \operatorname { dist } \left( \mathcal { E } _ { k } \left( \overline { Y } _ { k } , Z _ { k + 1 } \right) , \Gamma _ { k + 1 } \right) ^ { 2 } \right] \end{aligned}
\end{align}
for $Z _ { k + 1 } \sim \mathcal { N } \left( 0 ,  I _ { q } \right)$, and we aim at finding a grid $\Gamma_{k+1}^{\star}$ that minimizes the distortion function.  For a given size of the grid, it is known that an optimal quantizer exists (see e.g. \cite{GrafLu2000}). Moreover, in the one-dimensional case, if the density of the random variable to be discretized is absolutely continuous and log-concave, then the optimal quantizer is unique.
\begin{remark}
	The conditional distribution of the  Euler process $\overline { Y }$ is Gaussian. Also, each component of a Gaussian vector is Gaussian. 
\end{remark}
To quantize the vector $\overline { Y }_k\in\RR^d$, \cite{Fiorin2018} consider each component of the vector separately: they quantize $\overline { Y } _ { k } ^ { \ell }$ over a grid $\Gamma _ { k } ^ { \ell }$ of size $N _ { k } ^ { \ell }$ for $\ell = 1 , \ldots , d$ and then they define its product quantization $\widehat { Y } _ { k }$, i.e., the quantizer of the whole vector, on the product grid $\Gamma _ { k } = \bigotimes _ { \ell = 1 } ^ { d } \Gamma _ { k } ^ { \ell }$ of size $N _ { k } = N _ { k } ^ { 1 } \times \ldots \times N _ { k } ^ { d }$ as $\widehat { Y } _ { k } = \left( \widehat { Y } _ { k } ^ { 1 } , \ldots , \widehat { Y } _ { k } ^ { d } \right)$.

More precisely, for any $k = 0 , \ldots , n$ and any given $\ell \in \{ 1 , \ldots , d \}$, $\widehat { Y } _ { k } ^ { \ell }$ denotes the quantization of $\overline { Y } _ { k } ^ { \ell }$ on the grid $\Gamma _ { k } ^ { \ell } = \left\{ y_ { k } ^ { \ell , i _ { \ell } } , i _ { \ell } = 1 , \ldots , N _ { k } ^ { \ell } \right\}$. The idea of \cite{Fiorin2018} is as follows: assume we have access to $\Gamma _ { k } ^ { \ell }$, an $N _ { k } ^ { \ell }$-quantizer, for $\ell = 1 , \ldots , d$, of the $\ell$-th component $\overline { Y } _ { k } ^ { \ell }$ of $\overline { Y } _ { k }$. They define a componentwise recursive product quantizer $\Gamma _ { k } = \bigotimes _ { \ell = 1 } ^ { d } \Gamma _ { k } ^ { \ell }$ of size $N _ { k } = N _ { k } ^ { 1 } \times \ldots \times N _ { k } ^ { d }$ of the vector $\overline { Y } _ { k } = \left( \overline { Y } _ { k } ^ { \ell } \right) _ { \ell = 1 , \ldots , d }$ via
\begin{align}
	\Gamma _ { k } = \left\{ \left( y _ { k } ^ { 1 , i _ { 1 } } , \ldots , y_ { k } ^ { d , i _ { d } } \right) , \quad y_ { k } ^ { \ell , i_\ell } \in \Gamma _ { k } ^ { \ell } \quad \text { for } \ell \in \{ 1 , \ldots , d \} \text { and } i _ { \ell } \in \left\{ 1 , \ldots , N _ { k } ^ { \ell } \right\} \right\}.
\end{align}
To leverage the conditional normality feature, suppose now that $\overline { Y } _ { k }$ has already been quantized and that we have the associated weights $\mathbb { P } \left( \widehat { Y } _ { k } = \mathbf {y_ { k } ^ { i}} \right) $, $ \mathbf i \in I _ { k }$, where $ \mathbf {y_ { k } ^ { i}} : = \left( y_ { k } ^ { 1 , i _ { 1 } } , \ldots , y_ { k } ^ { d , i _ { d } } \right),$ $ \mathbf i : = \left( i _ { 1 } , \ldots , i _ { d } \right) \in I _ { k }$ and
\begin{align}
	I _ { k } = \left\{ \left( i _ { 1 } , \ldots , i _ { d } \right) , i _ { \ell } \in \left\{ 1 , \ldots , N _ { k } ^ { \ell } \right\} \right\}, \ k \in \{ 0 , \ldots , n \}.
\end{align}
By setting $\widetilde { Y } _ { k } ^ { \ell } = \mathcal { E } _ { k } ^ { \ell } \left( \widehat { Y } _ { k } , Z _ { k + 1 } \right)$, one can approximate each component of $\overline { D } _ { k + 1 } \left( \Gamma _ { k + 1 } \right)$ via $\widetilde { D } _ { k + 1 } ^ { \ell } \left( \Gamma _ { k + 1 } ^ { \ell } \right)$, $\ell=1,\ldots, d$
where
\begin{align}
	\begin{aligned}
		\widetilde { D } _ { k + 1 } ^ { \ell } \left( \Gamma _ { k + 1 } ^ { \ell } \right) & : = \mathbb { E } \left[ \operatorname { dist }\left( \widetilde { Y } _ { k + 1 } ^ { \ell } , \Gamma _ { k + 1 } ^ { \ell } \right) ^ { 2 } \right]  = \mathbb { E } \left[ \operatorname { dist } \left( \mathcal { E } _ { k } ^ { \ell } \left( \widehat { Y } _ { k } , Z _ { k + 1 } \right), \Gamma _ { k + 1 } ^ { \ell } \right) ^ { 2 } \right] \\ & = \sum _ {\mathbf  i \in I _ { k } } \mathbb { E } \left[ \operatorname { dist } \left( \mathcal { E } _ { k } ^ { \ell } \left( \mathbf{y_ { k } ^ { i }} , Z _ { k + 1 } \right) , \Gamma _ { k + 1 }^{\ell} \right) ^ { 2 } \right] \mathbb { P } \left( \widehat { Y } _ { k } =\mathbf{y_ { k } ^ { i }} \right) .
	\end{aligned}
\end{align}
Such approximation allows \cite{Fiorin2018} to introduce the sequence of product recursive quantizations of $\widehat { Y }=\left( \widehat { Y } _ { k } \right) _ { k = 0 , \cdots , n}$, for $k=0,\ldots n-1$, as
\begin{align}
	\label{eq:quantizationAlgorithm}
	\begin{cases}
		\widetilde { Y } _ { 0 } = \widehat { Y } _ { 0 } = y_0 , \quad \widehat { Y } _ { k } = \left( \widehat { Y } _ { k } ^ { 1 }, \ldots , \widehat { Y } _ { k } ^ { d } \right) \\
		\widehat { Y } _ { k } ^ { \ell } = \operatorname { Proj } _ { \Gamma _ { k } ^ { \ell } } \left( \widetilde { Y } _ { k } ^ { \ell } \right)  \quad \text { and } \quad \widetilde { Y } _ { k + 1 } ^ { \ell } = \mathcal { E } _ { k } ^ { \ell } \left( \widehat { Y } _ { k } , Z _ { k + 1 } \right) , \ell = 1 , \ldots , d\\
		\mathcal { E } _ { k } ^ { \ell } ( y , z ) = y^{\ell} + \Delta b^{\ell} (y) + \sqrt { \Delta } \left( \sigma ^ { \ell \bullet } \left(  y \right) | z \right) , z = \left( z ^ { 1 } , \ldots , z ^ { q } \right) \in \mathbb { R } ^ { q }\\
		y = \left( y ^ { 1 } , \ldots , y ^ { d } \right) , b = \left( b ^ { 1 } , \ldots , b ^ { d } \right) \text { and } \left( \sigma ^ { \ell  \bullet } \left( y \right) | z \right) = \sum _ { m = 1 } ^ { q } \sigma ^ { \ell m } \left( y \right) z ^ { m }
	\end{cases}
\end{align}
where for every matrix $A \in \mathcal M(d,q), a^{\ell \bullet} = {[a_{\ell j}]}_{j = 1, \dots, q}$.

We conclude this section by providing some information about the computation of transition probabilities. In \cite{Fiorin2018} three methods are proposed: the first one, in their Proposition 3.1, concerns the computation of transition probabilities of the whole vector $\widehat { Y }$ 
, the second covers the case where the diffusion matrix is diagonal and the third is a corollary to their Proposition 3.1 for the case where we are interested only in one component of the whole vector. We refer the reader to this paper for all the details on how to instantaneously compute these probabilities once the quantization grids have been obtained.

\section{Computing the conditional expectations}\label{sec:conditional expectations}
Our numerical scheme \eqref{our_scheme} has been conceived in such a way that computing the conditional expectations, with respect to ${(\mathcal F_{t_k})}_{k=0, \dots, n}$, only requires the knowledge of the stochastic process $\overline Y$. We will see that this results, from the practical point of view, in a handy, easy to understand and ready-to-use numerical scheme.

Before proceeding, we now rigorously prove that for every $k=0, \dots, n-1$
$$
\widetilde U_{t_k} = u_k(\overline Y_{t_k}) \quad \textrm{and} \quad \widetilde V_{t_k} = v_k(\overline Y_{t_k}) 
$$
for given Borel functions $u_k: \mathbb R^d \rightarrow \mathbb R$ and $v_k:\mathbb R^d \rightarrow \mathbb R^q$.

\begin{proposition}\label{U_V_u_v}
	For every $k \in \{0, \dots, n-1  \}$ the update rule for the control satisfies $\widetilde V_{t_k} = v_k(\overline Y_{t_k})$ and for every $l \in \{ 0, \dots, n \}$ we have $\widetilde U_{t_l} = u_l (\overline Y_{t_l})$, where $v_k: \mathbb R^d \rightarrow \mathbb R^q$ and $u_l :\mathbb R^d \rightarrow \mathbb R$ are
	\begin{equation*}
		\left\{
		\begin{array}{l}
			v_k (y)  :=  \frac{1}{\Delta} {(\sigma \left( y \right)^{\top})}^{-1} [ g_{1,k}(y) + \Delta \cdot g_{2,k}( y)  b \left( y \right)], \quad k=0, \dots, n-1\\
			u_n (y) :=  h(y), \quad \textrm{and} \quad u_l (y) := g_{3,l} (y) +  \Delta  \cdot f \left( {t_{l}},y , g_{3,l}(y),  v_l(y) \right), \quad l=0, \dots, n-1
		\end{array}
		\right.
	\end{equation*}
	with $g_{1,k}: \mathbb R^d \rightarrow \mathbb R^d$, $g_{2,k}: \mathbb R^d \rightarrow \mathbb R$ and $g_{3,l} : \mathbb R^d \rightarrow \mathbb R$ as follows 
	\begin{equation*}
		\left\{
		\begin{array}{rcl}
			g_{1,k}(y) &:=& {\mathbb E \left[ u_k \left(  \mathcal E_{k-1}( y , Z_k) \right) \left[  \mathcal E_{k-1} ( y , Z_k) - y \right] \right] } \\
			g_{2,k} (y) & := & \mathbb E \left[  u_k(  \mathcal E_{k-1}(y , Z_k) ) \right]\\
			g_{3,l}(y) & := & \mathbb E \left[   u_{l+1} ( \mathcal E_{l}(y, Z_{l+1})) \right]
		\end{array}
		\right.
	\end{equation*}
	and for $Z_j$'s i.i.d., $Z_j \sim \mathcal N(0, I_q)$.
\end{proposition}
\begin{proof}
	First of all notice that, by definition of $\widetilde U_{t_n}$, we immediately have $\widetilde U_{t_n} = h(\overline Y_{t_n}) = : u_n(\overline Y_{t_n})$.
	We now work on the control at time $t_{n-1}$. By recalling (cf. Section \ref{sec:scheme_U}) that $(t_{k+1} - t_k) = \Delta$ we find:
	\begin{eqnarray*}
		\widetilde{V}_{t_{n-1}} &=& \frac{1}{\Delta}  \left[ \sigma \left( \overline{Y} _ { t_{n-1} } \right)^{\top} \right]^{-1}  \Excond{}{\widetilde{U}_{t_{n}}\left(\overline{Y}_{t_{n}}-\overline{Y}_{t_{n-1}}\right)}{\cF_{t_{n-1}}} -  \left[ \sigma \left( \overline{Y} _ { t_{n-1} } \right)^{\top} \right]^{-1}  \Excond{}{\widetilde{U}_{t_{n}}}{\cF_{t_{n-1}}}b \left( \overline{Y} _ { t_{n-1} } \right)\\
		& = & \frac{1}{\Delta }  \left[ \sigma \left( \overline{Y} _ { t_{n-1} } \right)^{\top} \right]^{-1} \Excond{}{u_n(\overline Y_{t_n}) \left(\overline{Y}_{t_{n}}-\overline{Y}_{t_{n-1}}\right)}{\cF_{t_{n-1}}}\\
		& \quad &   +  \left[ \sigma \left( \overline{Y} _ { t_{n-1} } \right)^{\top} \right]^{-1}  \Excond{}{u_n(\overline Y_{t_n})}{\cF_{t_{n-1}}}b \left( \overline{Y} _ { t_{n-1} } \right).
	\end{eqnarray*}
	Now, from Equation \eqref{Euler_op}, we have $\overline{Y}_{t_{k+1}} = \mathcal E_{k}(\overline Y_{t_{k}}, Z_{k+1})$, where $Z_{k+1} \sim \mathcal N(0, I_q)$ and the $Z_k$'s, $k=0, \dots, n-1$ are i.i.d., and so
	\begin{eqnarray*}
		\widetilde V_{t_{n-1}} & = &  \frac{1}{\Delta }  \left[ \sigma \left( \overline{Y} _ { t_{n-1} } \right)^{\top} \right]^{-1} {\mathbb E \left\{ u_n \left(  \mathcal E_{n-1}(\overline Y_{t_{n-1}}, Z_n) \right) \left[  \mathcal E_{n-1} (\overline Y_{t_{n-1}}, Z_n) - \overline{Y}_{t_{n-1}}\right] | \mathcal F_{t_{n-1}} \right\} } \\
		& & +  \left[ \sigma \left( \overline{Y} _ { t_{n-1} } \right)^{\top} \right]^{-1} \Excond{}{u_n(  \mathcal E_{n-1}(\overline Y_{t_{n-1}}, Z_n) )}{\cF_{t_{n-1}}}b \left( \overline{Y} _ { t_{n-1} } \right) \\
		& = &  \frac{1}{\Delta }  \left[ \sigma \left( \overline{Y} _ { t_{n-1} } \right)^{\top} \right]^{-1}  \left[ g_{1,n}(\overline Y_{t_{n-1}}) + \Delta \ g_{2,n}( \overline{Y} _ { t_{n-1}})  b \left( \overline{Y} _ { t_{n-1} } \right) \right] =: v_{n-1}(\overline Y_{t_{n-1}}),
	\end{eqnarray*}
	where we have used the fact that ${(\overline Y_{t_j })}_{j=0, \dots, n} $ is a  Markov process (see e.g. \cite[Section 3.2.1]{ps2018}) measurable w.r.t $\mathcal F_{t_{n-1}}$ and $Z_n$ is independent of $\mathcal F_{t_{n-1}}$ and, for $y \in \mathbb R^d$, $g_{1,n}(y)$ and $g_{2,n}(y)$ are defined as in the statement.
	
	We now proceed by induction (notice that the values of $k$ and $l$ for which the claims hold are shifted): we assume that $\widetilde U_{t_{k+1}} = u_{k+1}(\overline Y_{t_{k+1}})$ and $\widetilde V_{t_k} = v_k(\overline Y_{t_k})$ and we prove that this implies that $\widetilde U_{t_{k}} = u_{k}(\overline Y_{t_{k}})$ and $\widetilde V_{t_{k-1}} = v_{k-1}(\overline Y_{t_{k-1}})$.\\ 
	The proof on $\widetilde V_{t_{k-1}}$ is analogous to what we have just done for $\widetilde V_{t_{n-1}}$, so we omit it. It remains the $\widetilde U_{t_k}$-part, which we develop now:
	\begin{eqnarray*}
		\widetilde U_{t_{k}} & = & \Excond{}{\widetilde{U}_{t_{k+1}}}{\cF_{t_k}}+ \Delta  \ f \left( {t_{k}},\overline{Y}_{t_{k}},\Excond{}{\widetilde{U}_{t_{k+1}}}{\cF_{t_k}},\widetilde{V}_{t_{k}} \right) \\
		&=&  \Excond{}{ u_{k+1} (\overline Y_{t_{k+1}})  }{\cF_{t_k}}+ \Delta \ f \left( {t_{k}},\overline{Y}_{t_{k}},\Excond{}{ u_{k+1} (\overline Y_{t_{k+1}})  }{\cF_{t_k}},  v_k(\overline Y_{t_k}) \right)\\
		& = &  \Excond{}{  u_{k+1} ( \mathcal E_{k}(\overline Y_{t_{k}}, Z_{k+1})) }{\cF_{t_k}}+ \Delta \ f \left( {t_{k}},\overline{Y}_{t_{k}},\Excond{}{ u_{k+1} ( \mathcal E_{k}(\overline Y_{t_{k}}, Z_{k+1}) )  }{\cF_{t_k}},  v_k(\overline Y_{t_k}) \right)\\
		& = & g_{3,k} (\overline Y_{t_k}) +  \Delta \ f \left( {t_{k}},\overline{Y}_{t_{k}}, g_{3,k}(\overline Y_{t_k}),  v_k(\overline Y_{t_k}) \right) =: u_k(\overline Y_{t_k}),
	\end{eqnarray*}
	where we have used the functions $g_{3,k}$ introduced in the statement.
\end{proof}

In summary, we can compute the conditional expectations in the discretization scheme \eqref{our_scheme} as follows, by exploiting Proposition \ref{U_V_u_v} and the Markovianity of the discrete time stochastic process ${(\overline Y_{t_k })}_{k=0, \dots, n} $
\begin{equation*}
	\left\{
	\begin{array}{rcl}
		\mathbb E [\widetilde U_{t_{k+1}} | \mathcal F_{t_k}] & = & \mathbb E [u_{k+1} (\overline Y_{t_{k+1}}) | \mathcal F_{t_k}]  =  \mathbb E [u_{k+1} (\overline Y_{t_{k+1}}) | \overline Y_{t_{k}}] \\
		\mathbb E [\widetilde{U}_{t_{k+1}} \left(\overline{Y}_{t_{k+1}}-\overline{Y}_{t_k} \right) | \mathcal F_{t_k}] &=& \mathbb E [u_{k+1}(\overline{Y}_{t_{k+1}}) \left(\overline{Y}_{t_{k+1}}-\overline{Y}_{t_k} \right) | \mathcal F_{t_k}]\\
		& = &   \mathbb E [u_{k+1}(\overline{Y}_{t_{k+1}}) \left(\overline{Y}_{t_{k+1}}-\overline{Y}_{t_k} \right) | \overline{Y}_{t_k} ].
	\end{array}
	\right.
\end{equation*}

\subsection{Approximation via quantization}
As a final step, now we approximate $\overline Y$ via $\widehat Y$, which is obtained as explained in Section \ref{sec:quantization} and we get the quantized final version of the recursive discretization scheme \eqref{our_scheme} (we recall that, for every $k=0, \dots, n$, $\widehat Y_{t_k}$ is the quantization of $\overline Y_{t_k}$):
\begin{equation}
	\label{our_scheme_quant}
	\left\{
	\begin{array}{rcl}
		\widehat{U}_{t_n} & = & h(\widehat{Y}_{t_{n}}) \quad \textrm{and for} \quad k=0, \dots, n-1  \\
		\widehat{U}_{t_k} & = &  \mathbb E \left[ \widehat{U}_{t_{k+1}} | \widehat Y_{t_{k}} \right]  + \Delta \ f \left( {t_{k}},\widehat{Y}_{t_{k}},\mathbb E [ \widehat{U}_{t_{k+1}} | \widehat Y_{t_{k}}] ,\widehat{V}_{t_{k}} \right) =: \widehat u_{k}( \widehat Y_{t_{k}})\\
		\widehat{V}_{t_k} &=& \frac{1}{\Delta}  \left[ \sigma \left( \widehat{Y} _ { t_k } \right)^{\top} \right]^{-1} \mathbb E \left[ \widehat{U}_{t_{k+1}} \left(\widehat{Y}_{t_{k+1}}-\widehat{Y}_{t_k} \right) | \widehat{Y}_{t_k} \right]  \\
		& & -  \left[ \sigma \left( \widehat{Y} _ { t_k } \right)^{\top} \right]^{-1}  \mathbb E \left[  \widehat{U}_{t_{k+1}} | \widehat Y_{t_{k}} \right]  b \left( \widehat{Y} _ { t_{k} } \right) =: \widehat v_k( \widehat{Y}_{t_k}), \\
	\end{array}
	\right.
\end{equation}
with $ \widehat u_{k}: \Gamma_{k} \rightarrow \mathbb R$ and $ \widehat v_{k}:  \Gamma_{k} \rightarrow \mathbb R^q$ Borel functions, for $k=0, \dots, n-1$. 
Namely, the conditional expectations in \eqref{our_scheme} are approximated as:
\begin{equation} \label{eq:cond_exp_quant}
	\left\{
	\begin{array}{rl}
		\mathbb E [\widetilde U_{t_{k+1}} | \mathcal F_{t_k}] & =   \mathbb E [u_{k+1} (\overline Y_{t_{k+1}}) | \overline Y_{t_{k}}] \approx  \mathbb E \left[ \widehat{U}_{t_{k+1}} | \widehat Y_{t_{k}} \right] = \mathbb E [\widehat u_{k+1}( \widehat Y_{t_{k+1}}) | \widehat Y_{t_{k}}] \\
		&  \\
		\mathbb E [\widetilde{U}_{t_{k+1}} \left(\overline{Y}_{t_{k+1}}-\overline{Y}_{t_k} \right) | \mathcal F_{t_k}] & =    \mathbb E [u_{k+1}(\overline{Y}_{t_{k+1}}) \left(\overline{Y}_{t_{k+1}}-\overline{Y}_{t_k} \right) | \overline{Y}_{t_k} ]\\
		&   \approx \mathbb E [  \widehat{U}_{t_{k+1}} \left(\widehat{Y}_{t_{k+1}}-\widehat{Y}_{t_k} \right) | \widehat{Y}_{t_k} ] \\
		& =  \mathbb E [ \widehat u_{k+1}( \widehat Y_{t_{k+1}}) \left(\widehat{Y}_{t_{k+1}}-\widehat{Y}_{t_k} \right) | \widehat{Y}_{t_k} ]  .
	\end{array}
	\right.
\end{equation}
The approximating scheme is then fully explicit, since for every $ \widetilde \omega \in \Omega$ such that $\widehat Y_{t_k}(\widetilde \omega) = \mathbf{y_k^i} $, for a given $\mathbf i \in I_k$, we have
\begin{equation}
	\label{eq:cond_exp_quant_expl}
	\begin{cases}
		\mathbb E \left[ \widehat  u_{k+1} (\widehat Y_{t_{k+1}}) | \widehat Y_{t_{k}} \right]  (\widetilde \omega) \\
		\quad\quad\quad\quad\quad\quad\quad\quad=  \sum_{\mathbf j \in I_{k+1}} \widehat u_{k+1}(\mathbf{y_{k+1}^j}) \mathbb { P } \left( \widehat { Y } _ { t_{k + 1} } = \mathbf{y_{k+1}^j} | \widehat { Y } _ { t_k } =\mathbf{y_{k}^i}\right)\\
		\mathbb E \left[ \widehat u_{k+1}(\widehat{Y}_{t_{k+1}}) \left(\widehat{Y}_{t_{k+1}}-\widehat{Y}_{t_k} \right) | \widehat{Y}_{t_k} \right]  (\widetilde \omega)\\
		\quad\quad\quad\quad\quad\quad\quad\quad=  \sum_{\mathbf j \in I_{k+1}} \widehat u_{k+1}(\mathbf{y_{k+1}^j}) \left( \mathbf{y_{k+1}^j} - \mathbf{y_{k}^i} \right)  \mathbb { P } \left( \widehat { Y } _ { t_{k + 1} } = \mathbf{y_{k+1}^j} | \widehat { Y } _ { t_k } =\mathbf{y_{k}^i}\right).
	\end{cases}
\end{equation}

\begin{remark}
	As expected and already announced, the proposed discretization scheme is fully driven by the process $\widehat Y$, which is the quantization of the Euler scheme process $\overline Y$.
\end{remark}

In the next section, we provide a study of the numerical error associated with our scheme \eqref{our_scheme_quant}.

\section{The error}\label{sec:error}
The analysis of the error is divided in two steps: in Section \ref{sec:TimeError} we study  the error in time whereas Section \ref{sec:SpaceError} focuses on the space dimension.  Here $\overline Y$ denotes the Euler scheme relative to the stochastic process $Y$ and we will work under the following:
\begin{assumption}\label{ellipticity}
	We have $q=d$ and the matrix $\sigma \sigma^{\top}$ is uniformly elliptic, namely, for every $y \in \mathbb R^d$, denoting by $a_{ij}(y), i,j=1,\dots,d$ the elements of $[\sigma(y) \sigma(y)^{\top}]$, there exists $\lambda_0 > 0$ such that
	$$
	\frac{1}{\lambda_0} || \xi ||^2 \le	\sum_{i,j=1}^d a_{ij}(y) \xi_i \xi_j \le \lambda_0 || \xi ||^2, \quad \xi \in \mathbb R^d.
	$$
\end{assumption}

\begin{remark}\label{rem:Frob}
	a) The above Assumption  \ref{ellipticity} ensures that for every $y \in \mathbb R^d$, the matrix $\sigma (y) $ is positive definite, hence invertible, and bounded. The inverse matrix $\sigma (y)^{-1}$ is also bounded. More precisely, denoting by ${|| \cdot ||}_F$ the Frobenius norm \footnote{Recall that given a matrix $B :=b_{ij}, i,j=1,\dots,d$, ${|| B ||}_F := \sqrt{\textrm{tr} (B B^{\top})} = \sqrt{\sum_{i,j} (b_{i,j})^2}$.}, we have that for every $y \in \mathbb R^d$
	$$
	{|| \sigma^{-1}(y) ||}_F^2 \le \lambda_0.
	$$
	This will be crucial in Section \ref{sec:SpaceError}.\\
	b) Assumption \ref{ellipticity}, together with a Lipschitz continuity condition on $h$ with Lipschitz constant $K$, is required by \cite[Lemma 2.5 (i)]{Zhang04} to prove that the control process $V$ admits a c\`adl\`ag version. This, in turn, is needed in \cite[Theorem 3.1]{Zhang04} to prove the following: 
	$$ 
	\sum_{i=1}^n \mathbb E \left\{ \int_{t_i}^{t_{i+1}} \left[ | V_t - V_{t_{i-1}}|^2 +  | V_t - V_{t_{i}}|^2 \right] dt \right\} \le C (1+|y_0|^2) \Delta , 
	$$ 
	where $C$ is a constant depending only on $T$ and $K$ and we recall that $\Delta=\Delta_n = \frac{T}{n}$.
	The results by Zhang are contained in \cite[Theorem 3.1]{ps2018}. We adapt them below, in Theorem \ref{thm:error_time}, to our setting.
\end{remark}

\subsection{Time discretization error}\label{sec:TimeError}
Studying the error of our scheme \eqref{our_scheme} with respect to time means computing a proper distance between $(U, V)$ and ${(\widetilde U, \widetilde V)}$. Inspired by \cite{ps2018}, we will adapt to our setting their Theorem 3.1.

Before stating the time discretization error result, we need to introduce, as typically done in the literature, the continuous time extension of ${(\widetilde U_{t_k})}_{k=0, \dots,n}$, denoted by ${(\widetilde U_{t})}_{t \in [0,T]}$.
We introduce the random variable 
$$
M_T := \sum_{k=1}^n  \widetilde U_{t_k} - \mathbb E \left[ \widetilde U_{t_k} | \mathcal F_{t_{k - 1}}\right]
$$
and we notice that $M_T \in L^2 $ and so, by the Martingale Representation Theorem, $M$ admits the following representation:
$$
M_T = \int_0^T \widetilde V_s^\top dW_s,
$$
for an ${(\mathcal F_t)}_{t \in [0,T]}$-progressively measurable stochastic process $\widetilde V$, with values in $\mathbb R^q$, such that $\mathbb E \left[  \int_0^T | \widetilde V_s|^2 ds \right] < \infty$. As a consequence, 
\begin{equation}\label{eq:defVtilde}
	\widetilde U_{t_k} - \mathbb E \left[ \widetilde U_{t_k} | \mathcal F_{t_{k - 1}}\right] = \int_{t_{k-1}}^{t_k} \widetilde V_s^\top dW_s
\end{equation}
and we introduce the continuous extension ${(\widetilde U_{t})}_{t \in [0,T]}$ as follows: if $t \in [t_k, t_{k+1})$,
\begin{equation}\label{eq:defU_tilde_cont}
	\widetilde U_{t} = \widetilde U_{t_k} - (t - t_k) f \left( t_k, \overline Y_{t_k}, \mathbb E\left[\widetilde U_{t_{k+1}} | \mathcal F_{t_k}\right], \widetilde V_{t_k} \right) + \int_{t_k}^t \widetilde V_s^\top dW_s .
\end{equation}
\begin{remark}
	Recalling Equation \eqref{scheme_ps} and exploiting Equation \eqref{eq:defVtilde}, we get, for every $k=0, \dots, n-1$,
	$$
	\widetilde{V}_{t_k}^{\textrm{PS}} =\frac{1}{\Delta} \Excond{}{\widetilde{U}_{t_{k+1}}\left( W_{t_{k+1}} - W_{t_k}\right)}{\cF_{t_k}} = \frac{1}{\Delta} \mathbb E \left[ \int_{t_k}^{t_ {k+1}} \widetilde V_s ds | \mathcal F_{t_k} \right].
	$$
	This is extensively used in the proof of \cite[Theorem 3.1]{ps2018}, which we mimic here, to prove the error bounds with respect to time.
\end{remark}

\begin{theorem} \label{thm:error_time}
	i) Under Assumption \ref{ass:U},  let $f: [0,T] \times \mathbb R^d \times \mathbb R \times \mathbb R^q \rightarrow \mathbb R$ be Lipschitz with respect to time and let $h: \mathbb R^d \rightarrow \mathbb R$ be Lipschitz. Then there exists a real constant $\overline C > 0$ only depending on $b, \sigma,f,T$ such that, for every $n \ge 1$
	$$
	\max_{k=0, \dots, n} \mathbb E \left[  | U_{t_k} - \widetilde U_{t_k}|^2  \right] + \int_0^T \mathbb E \left[  | V_t - \widetilde  V_t |^2 \right] dt \le \overline C \left( \Delta + \int_0^T \mathbb E \left[  | V_s - V_{\underline s} |^2 \right] ds  \right),
	$$ 
	where $\underline s = t_k $ if $s \in [t_k, t_{k+1})$.\\
	ii) Assume moreover that $b, \sigma$ and $f$ are continuously differentiable in their spatial variable with bounded partial derivatives and that $f$ is $\frac{1}{2}$-H\"older continuous with respect to time. Then the process $V$ admits a c\`adl\`ag modification and 
	$$
	\int_0^T \mathbb E \left[ | V_s - V_{\underline s} |^2  \right] ds  \le C' \Delta,
	$$
	for a real positive constant $C'$ (only depending on $b, \sigma, f, T$), so that we have
	$$
	\max_{k=0, \dots, n} \mathbb E \left[  | U_{t_k} - \widetilde U_{t_k}|^2 \right] + \int_0^T \mathbb E \left[ | V_t - \widetilde  V_t |^2 \right] dt \le \widetilde C \Delta,
	$$
	for a real positive constant $\widetilde C$ (only depending on $b,\sigma,f,T$).
\end{theorem}
\begin{proof}
	Part $ii)$ is a consequence of $i)$ and it is obtained via \cite[Lemma 2.5 (i) and Theorem 3.1]{Zhang04}.\\
	The proof of $i)$ is the same as the one in \cite[Theorem 3.1 a)]{ps2018} relatively to Steps 2 and 3, while something has to be made precise relatively to Step 1. 
	Indeed, the two schemes \eqref{scheme_ps} and \eqref{our_scheme} differ in the control discretization. Nevertheless, since here $\overline Y$ is the Euler scheme of $Y$, namely $\overline Y_{t_{k+1}} - \overline Y_{t_k}=  \Delta b (\overline Y_{t_k}) + \sigma(\overline Y_{t_k})^\top (W_{t_{k+1}}  - W_{t_k})$, $k=0, \dots, n-1$, 
also Step 1 in \cite[Theorem 3.1]{ps2018} can be retraced straightforwardly.
\end{proof}

\begin{remark} \label{rem:diff_Pages}
	Despite the fact that the time error bounds for our proposed scheme are the same as in \cite{ps2018}, our recursions only require the discretization of the process $\overline Y$ and this results in an increased numerical efficiency, namely in the speed-up of the computational time. What is more, in the approximation of the conditional expectations required in our scheme, we only need the transition probabilities of the quantized process $\widehat Y$, i.e., for $\mathbf i \in I_k, \mathbf j \in I_{k+1}$:
	$$
	\mathbb { P } \left( \widehat { Y } _ { t_{k + 1} } = \mathbf{y_{k+1}^j} | \widehat { Y } _ { t_k } =\mathbf{y_{k}^i}\right).
	$$
	This is not the case in \cite{ps2018}, where the authors, in their scheme \eqref{scheme_ps} for $\widetilde V^{PS}$, need to additionally compute $\mathbb E[ \widehat Y_{t_{k+1}} (W_{t_{k+1}} - W_{t_k}) | \mathcal F_{t_k}], k=0, \dots, n-1$. So, they have to estimate, for $\mathbf i \in I_k, \mathbf j \in I_{k+1}$,  the weights (we stick to their notation)
	$$
	\pi_{ij}^{W,k} := \frac{1}{\mathbb P( \widehat Y_{t_k} = \mathbf{y_k^i} ) } \mathbb E \left[  (W_{t_{k+1}} - W_{t_k})  \mathbf{1}_{\{\widehat Y_{t_{k+1}} = \mathbf{y_{k+1}^j}, \widehat Y_{t_k} = \mathbf{y_k^i} \}}
	\right],
	$$
	which is done via Monte Carlo simulation (see \cite[Section 5]{ps2018}), hence requiring additional numerical effort, which is not needed in our case.
\end{remark}

\subsection{Space discretization (quantization) error}\label{sec:SpaceError}
Here we study the quadratic quantization error induced by the approximation of $(\widetilde U_{t_k}, \widetilde V_{t_k})$ in \eqref{our_scheme} by $(\widehat U_{t_k}, \widehat V_{t_k})$ in \eqref{our_scheme_quant}, for every $k=0, \dots, n$.
We intuitively expect both the error components $|| \widetilde U_{t_k} - \widehat U_{t_k}  ||_2^2$ and $|| \widetilde V_{t_k} - \widehat V_{t_k} ||_2^2$ to be written as functions of the quantization error $|| \overline Y_{t_i} - \widehat Y_{t_i} ||_2^2$ for $i=k, \dots, n$. 

The numerical scheme relative to $U$ is the same as in Pag\`es and Sagna, so that the error component $|| \widetilde U_{t_k} - \widehat U_{t_k}  ||_2^2$ for $k=0, \dots, n$ reads as in \cite[Theorem 3.2 a), Equation (31)]{ps2018}. We recall this result here, for the reader's ease, in the following lemma:
\begin{lemma} 
Under Assumptions \ref{ass:Y} and \ref{ass:U} and assuming that $h$ is $[h]_{\textrm{Lip}}$-continuous, we have that for every $k=0, \dots, n$
\begin{equation} \label{eq:err_Utilde_Uhat}
|| \widetilde U_{t_k} - \widehat U_{t_k}  ||_2^2 \le \sum_{i=k}^n e^{( 1 + \cL_6)(t_i - t_k)} K_i(b,\sigma,T,f) || \overline Y_{t_i} - \widehat Y_{t_i}  ||_2^2 
\end{equation} 
where $K_n(b, \sigma,T,f) = [h]_{\textrm{Lip}}^2$ and for every $k=0, \dots, n-1$ the other $K_k$'s are provided in \cite[Theorem 3.2 a)]{ps2018}.
\end{lemma}
We now hence focus on $|| \widetilde V_{t_k} - \widehat V_{t_k} ||_2^2$.

\begin{theorem}
Under Assumptions \ref{ass:Y} and \ref{ellipticity} and if $b$ is continuously differentiable with bounded derivative, there exist positive constants $\widehat C$ and $\overline C$, only depending on $(\lambda_0,\cL_3,\cL_4)$, such that for every $k=0, \dots, n$:
\begin{equation}
||  \widetilde V_{t_k} - \widehat V_{t_k}  ||_{2}^2 \le   \frac{1}{\Delta} [\Psi_k]_{\textrm{Lip}}^2  || \overline Y_{t_k} - \widehat Y_{t_k}  ||_{2}^2 +\left(  \frac{\widehat C}{\Delta}   +  \overline C  \right)  ||   \widetilde{U}_{t_{k+1}} -  \widehat{U}_{t_{k+1}}    ||_{2}^2 
\end{equation}
where $[\Psi_k]_{\textrm{Lip}}$ is the Lipschitz constant of the function $\Psi_k (x): \mathbb R^d \rightarrow  \mathbb R, \Psi_k (x):= \mathbb E [ u_{k+1}\left(  \mathcal { E } _ { k } (  x , Z) \right)\cdot   Z]$ for $Z \sim \mathcal N(0, I_q)$.
\end{theorem}
\begin{proof}
We start by noticing that 
\begin{eqnarray*}
{\big|\big| \widetilde V_{t_k} - \widehat V_{t_k}  \big| \big|}_{2}^2 & = &  { \big| \big| \widetilde V_{t_k} - \mathbb E \left[ \widetilde V_{t_k} |  \widehat Y_{t_k} \right] + \mathbb E \left[ \widetilde V_{t_k} |  \widehat Y_{t_k} \right]   -  \widehat V_{t_k}   \big| \big|}_{2}^2 \\
& \le &  \underbrace{{  \big| \big| \widetilde V_{t_k} - \mathbb E \left[ \widetilde V_{t_k} |  \widehat Y_{t_k} \right] \  \big| \big| }_{2}^2 }_{(I)}+ \underbrace{ {  \big| \big|  \mathbb E \left[ \widetilde V_{t_k} |  \widehat Y_{t_k} \right]   -  \widehat V_{t_k}   \big| \big| }_{2}^2}_{(II)}.
\end{eqnarray*}
Let us focus on $(I)$. By exploting Equations \eqref{eq:EulerScheme}, \eqref{Euler_op} and Proposition \ref{U_V_u_v}, we find 
\begin{eqnarray*}
\widetilde{V}_{t_k} & = & \frac{1}{\Delta} \Excond{}{\widetilde{U}_{t_{k+1}}\left( W_{t_{k+1}} - W_{t_k}\right)}{\cF_{t_k}} \\
& = & \frac{1}{\Delta} \Excond{}{ u_{k+1}\left(  \mathcal { E } _ { k } (  \overline{Y} _ { t_{k} } , Z_{k+1} ) \right)  \sqrt{\Delta} Z_{k+1} }{\cF_{t_k}}
\end{eqnarray*}
where we used $\left( W_{t_{k+1}} - W_{t_k}\right) =: \sqrt{\Delta}  Z_{k+1} $, with $Z_{k+1} \sim \mathcal N(0, I_q)$ independent of $\cF_{t_k}$. 
So, we have, being $\sigma(\widehat Y_{t_k}) \subset \sigma(\overline Y_{t_k}) \subset \mathcal F_{t_k}$ and using the tower property of conditional expectation:
\begin{eqnarray*}
{  \big| \big| \widetilde V_{t_k} - \mathbb E \left[ \widetilde V_{t_k} |  \widehat Y_{t_k} \right] \  \big| \big| }_{2}^2 &=& 
\scriptstyle \bigg| \bigg|   \frac{1}{\Delta}  \Excond{}{ u_{k+1}\left(  \mathcal { E } _ { k } (  \overline{Y} _ { t_{k} } , Z_{k+1} ) \right)  \sqrt{\Delta} Z_{k+1}  }{\cF_{t_k}}  -   \frac{1}{\Delta}  \Excond{}{  u_{k+1}\left(  \mathcal { E } _ { k } (  \overline{Y} _ { t_{k} } , Z_{k+1} ) \right)  \sqrt{\Delta} Z_{k+1} }{ \widehat Y_{t_k}}   \bigg| \bigg|_{2}^2 \\
& \le &  \frac{1}{\Delta} \bigg| \bigg|   \Psi_k \left(  \overline{Y} _ { t_{k}}  \right)   -     \Psi_k \left(  \widehat{Y} _ { t_{k}}  \right)  \bigg| \bigg|_{2}^2  \le \frac{1}{\Delta} [\Psi_k]_{\textrm{Lip}}^2 {  \big| \big| \overline Y_{t_k} - \widehat Y_{t_k}  \big| \big| }_{2}^2
\end{eqnarray*}
where in the second passage we defined $\Psi_k (x) = \mathbb E [ u_{k+1}\left(  \mathcal { E } _ { k } (  x , Z_{k+1} ) \right)  Z_{k+1}]$, we used the definition of conditional expectation as best $L^2$-approximation and exploited the Lipschitzianity of $\Psi_k$, for which we refer to \cite[Prop. 3.4 (b)]{ps2018} (therein $\Psi_k$ corresponds to $z_k$).

Now, consider $(II)$: since $\sigma(\widehat Y_{t_k}) \subset \sigma(\overline Y_{t_k}) \subset \mathcal F_{t_k}$ and by recalling the definition of $\widetilde V_{t_k}$ and $\widehat V_{t_k}$ in Equations \eqref{our_scheme} and \eqref{our_scheme_quant} we have
\begin{eqnarray*}
{  \big| \big|  \mathbb E \left[ \widetilde V_{t_k} |  \widehat Y_{t_k} \right]   -  \widehat V_{t_k}   \big| \big| }_{2}^2 & = & {  \big| \big|  \mathbb E \left[ \widetilde V_{t_k} -  \widehat V_{t_k}  |  \widehat Y_{t_k} \right]     \big| \big| }_{2}^2 \\
& = &   \bigg| \bigg|  \mathbb E \Bigg[  \frac{1}{\Delta}  \left[ \sigma \left( \overline{Y} _ { t_k } \right)^{\top} \right]^{-1}  \widetilde{U}_{t_{k+1}}\left(\overline{Y}_{t_{k+1}}-\overline{Y}_{t_k}\right) -  \frac{1}{\Delta}  \left[ \sigma \left( \widehat{Y} _ { t_k } \right)^{\top} \right]^{-1}  \widehat{U}_{t_{k+1}} \left( \widehat{Y}_{t_{k+1}} - \widehat{Y}_{t_k}\right)   \\
& & + \left[ \sigma \left( \overline{Y} _ { t_k } \right)^{\top} \right]^{-1} \widetilde{U}_{t_{k+1}} b \left( \overline{Y} _ { t_{k} } \right)  -  \left[ \sigma \left( \widehat{Y} _ { t_k } \right)^{\top} \right]^{-1}  \widehat{U}_{t_{k+1}} b \left( \widehat{Y} _ { t_{k} } \right)  \bigg|  \widehat Y_{t_k} \Bigg]     \bigg| \bigg|_{2}^2 \\
& \le & \frac{2}{\Delta^2}   \underbrace{ \bigg| \bigg| \mathbb E \Bigg[ \left[ \sigma \left( \overline{Y} _ { t_k } \right)^{\top} \right]^{-1}  \widetilde{U}_{t_{k+1}}\left(\overline{Y}_{t_{k+1}}-\overline{Y}_{t_k}\right) -  \left[ \sigma \left( \widehat{Y} _ { t_k } \right)^{\top} \right]^{-1}  \widehat{U}_{t_{k+1}} \left( \widehat{Y}_{t_{k+1}} - \widehat{Y}_{t_k}\right) |  \widehat Y_{t_k} \Bigg]  \bigg| \bigg|_{2}^2}_{(IIa)} \\
& & + 2   \underbrace{ \bigg| \bigg| \mathbb E \Bigg[   \left[ \sigma \left( \overline{Y} _ { t_k } \right)^{\top} \right]^{-1} \widetilde{U}_{t_{k+1}} b \left( \overline{Y} _ { t_{k} } \right)  -  \left[ \sigma \left( \widehat{Y} _ { t_k } \right)^{\top} \right]^{-1}  \widehat{U}_{t_{k+1}} b \left( \widehat{Y} _ { t_{k} } \right)  \bigg|  \widehat Y_{t_k} \Bigg]    \bigg| \bigg|_{2}^2 }_{(IIb)} .
\end{eqnarray*}
We are hence led to focus now on $(IIa)$ and $(IIb)$. We start by $(IIb)$. By exploiting again the definition of conditional expectation with respect to $\widehat Y_{t_k}$ we find:
\begin{eqnarray*}
& \bigg| \bigg| \mathbb E \Bigg[   \left[ \sigma \left( \overline{Y} _ { t_k } \right)^{\top} \right]^{-1} \widetilde{U}_{t_{k+1}} b \left( \overline{Y} _ { t_{k} } \right)  -  \left[ \sigma \left( \widehat{Y} _ { t_k } \right)^{\top} \right]^{-1}  \widehat{U}_{t_{k+1}} b \left( \widehat{Y} _ { t_{k} } \right)  \bigg|  \widehat Y_{t_k} \Bigg]    \bigg| \bigg|_{2}^2  &\\
& \le  \bigg| \bigg|   \left[ \sigma \left( \widehat {Y} _ { t_k } \right)^{\top} \right]^{-1}  \mathbb E \left[ \widetilde{U}_{t_{k+1}}  \bigg|  \widehat Y_{t_k} \right]  b \left( \widehat{Y} _ { t_{k} } \right)  -   \left[ \sigma \left( \widehat{Y} _ { t_k } \right)^{\top} \right]^{-1}  \mathbb E \Bigg[ \widehat{U}_{t_{k+1}}  \bigg|  \widehat Y_{t_k}  \Bigg]   b \left( \widehat{Y} _ { t_{k} } \right)   \bigg| \bigg|_{2}^2 & \\
& \le  {|| (\sigma(\cdot)^{\top})^{-1} ||}_F^2  \ \bigg| \bigg|   \mathbb E \left[ \widetilde{U}_{t_{k+1}} -  \widehat{U}_{t_{k+1}}   \bigg|  \widehat Y_{t_k} \right]  b \left( \widehat{Y} _ { t_{k} } \right)   \bigg| \bigg|_{2}^2  \le  \lambda_0 \bigg| \bigg|   \mathbb E \left[ \widetilde{U}_{t_{k+1}} -  \widehat{U}_{t_{k+1}}   \bigg|  \widehat Y_{t_k} \right]  b \left( \widehat{Y} _ { t_{k} } \right)   \bigg| \bigg|_{2}^2 &
\end{eqnarray*}
where we have used the fact that ${|| A||}_2 \le {|| A||}_F$ for any  matrix $A$ and Remark \ref{rem:Frob} on the boundedness of the norm of $\sigma(\cdot)^{-1}$. Now notice that the boundedness of the derivative of $b$ implies its uniform continuity and so, since the quantizer $\widehat Y_{t_k}$ takes values on a compact set, $b(\widehat Y_{t_k})$ is also bounded. Namely, there exists $\bar c >0$, {\color{black}{only depending on $\cL_3$ and $\cL_4$,}} such that ${||  b \left( \widehat{Y} _ { t_{k} } \right)   ||}_2^2 \le \bar c ^2$ and we find
\begin{eqnarray*}
\lambda_0  \ \bigg| \bigg|   \mathbb E \left[ \widetilde{U}_{t_{k+1}} -  \widehat{U}_{t_{k+1}}   \bigg|  \widehat Y_{t_k} \right]  b \left( \widehat{Y} _ { t_{k} } \right)   \bigg| \bigg|_{2}^2 & \le & \lambda_0 \ \bar c^2  {  \bigg| \bigg|   \mathbb E \left[ \widetilde{U}_{t_{k+1}} -  \widehat{U}_{t_{k+1}}   \bigg|  \widehat Y_{t_k} \right]   \bigg| \bigg|}_{2}^2 \\
& = &  \lambda_0 \ \bar c^2  {  \bigg| \bigg|   \widetilde{U}_{t_{k+1}} -  \widehat{U}_{t_{k+1}}    \bigg| \bigg|}_{2}^2
\end{eqnarray*}
where the error ${  \big| \big|   \widetilde{U}_{t_{k+1}} -  \widehat{U}_{t_{k+1}}    \big| \big|}_{2}^2$ has already been studied in \cite{ps2018} and for this we refer to Equation \eqref{eq:err_Utilde_Uhat}.

We now move to the last term, $(IIa)$. Using again the definition of conditional expectation  with respect to $\widehat Y_{t_k}$ and the fact that the Frobenius norm of the matrix  $\sigma(\cdot)^{-1}$ is bounded (see Remark \ref{rem:Frob}), we have:
\begin{eqnarray*}
& \bigg| \bigg| \mathbb E \Bigg[ \left[ \sigma \left( \overline{Y} _ { t_k } \right)^{\top} \right]^{-1}  \widetilde{U}_{t_{k+1}}\left(\overline{Y}_{t_{k+1}}-\overline{Y}_{t_k}\right) -  \left[ \sigma \left( \widehat{Y} _ { t_k } \right)^{\top} \right]^{-1}  \widehat{U}_{t_{k+1}} \left( \widehat{Y}_{t_{k+1}} - \widehat{Y}_{t_k}\right) |  \widehat Y_{t_k} \Bigg]  \bigg| \bigg|_{2}^2  &\\
& \le   \bigg| \bigg|  \left[ \sigma \left( \widehat{Y} _ { t_k } \right)^{\top} \right]^{-1} \mathbb E \Bigg[  \widetilde{U}_{t_{k+1}}\left(\overline{Y}_{t_{k+1}}-\widehat{Y}_{t_k}\right) |  \widehat Y_{t_k} \Bigg]   -  \left[ \sigma \left( \widehat{Y} _ { t_k } \right)^{\top} \right]^{-1}  \mathbb E \Bigg[   \widehat{U}_{t_{k+1}} \left( \widehat{Y}_{t_{k+1}} - \widehat{Y}_{t_k}\right) |  \widehat Y_{t_k} \Bigg]  \bigg| \bigg|_{2}^2  &\\
& \le 
\lambda_0   \bigg| \bigg|  \mathbb E \left[  \mathbb E \left[   \widetilde{U}_{t_{k+1}}\left(\overline{Y}_{t_{k+1}}-\widehat{Y}_{t_k} \right)  |  \widehat Y_{t_{k+1}} \right] |  \widehat Y_{t_k} \right]   -  \mathbb E \left[   \widehat{U}_{t_{k+1}} \left( \widehat{Y}_{t_{k+1}} - \widehat{Y}_{t_k}\right) |  \widehat Y_{t_k} \right]  \bigg| \bigg|_{2}^2  &
\end{eqnarray*}
where in the last equality we used $\sigma(\widehat Y_{t_k}) \subseteq \sigma(\widehat Y_{t_{k+1}})$. Now, by definition of conditional expectation with respect to $\widehat Y_{t_{k+1}}$ we find:
\begin{eqnarray*}
& \lambda_0   \bigg| \bigg|  \mathbb E \left[  \mathbb E \left[   \widetilde{U}_{t_{k+1}}\left(\overline{Y}_{t_{k+1}}-\widehat{Y}_{t_k} \right]  |  \widehat Y_{t_{k+1}} \right) |  \widehat Y_{t_k} \right]   -  \mathbb E \left[   \widehat{U}_{t_{k+1}} \left( \widehat{Y}_{t_{k+1}} - \widehat{Y}_{t_k}\right) |  \widehat Y_{t_k} \right]  \bigg| \bigg|_{2}^2  &\\
& \le 
 \lambda_0  \  \bigg| \bigg|  \mathbb E \left[ \left( \widetilde{U}_{t_{k+1}} -  \widehat{U}_{t_{k+1}} \right) \left(  \widehat{Y}_{t_{k+1}} - \widehat{Y}_{t_k} \right) |  \widehat Y_{t_k} \right]  \bigg| \bigg|_{2}^2 
 \le \lambda_0 \bigg| \bigg|   \widetilde{U}_{t_{k+1}} -  \widehat{U}_{t_{k+1}}    \bigg| \bigg|_{2}^2   \cdot  \bigg| \bigg|    \widehat{Y}_{t_{k+1}} - \widehat{Y}_{t_k} \bigg| \bigg|_{2}^2 &
\end{eqnarray*}
where in the last passage we have used conditional Cauchy-Schwartz inequality. Hence, it appears again the error term ${  \big| \big|   \widetilde{U}_{t_{k+1}} -  \widehat{U}_{t_{k+1}}    \big| \big|}_{2}^2$, for which we refer to Equation \eqref{eq:err_Utilde_Uhat}. It remains, then, to deal with the error ${  \big| \big|   \widehat{Y}_{t_{k+1}} - \widehat{Y}_{t_k} \big| \big|}_{2}^2$.

We begin this last part of the proof by recalling that, for every $k=0,\dots,N$, $\widehat Y_{t_k}$ is the stationary quantizer relative to $\overline Y_{t_k}$ obtained via recursive marginal quantization as explained in Section \ref{sec:quantization}. Namely, $\mathbb E(\overline Y_{t_k}| \widehat Y_{t_k}) = \widehat Y_{t_k} $ and so, via conditionl Jensen's inequality and the tower property, in case when $g$ is convex we have $\mathbb E [g(\widehat Y_{t_k})] \le \mathbb E [g(\overline Y_{t_k})]$. So when $g$ is the square function, we find
\begin{eqnarray*}
{  \big| \big|   \widehat{Y}_{t_{k+1}} - \widehat{Y}_{t_k} \big| \big|}_{2}^2 &=& \mathbb E \left[ (  \widehat{Y}_{t_{k+1}} - \widehat{Y}_{t_k} )^2 \right] =   \mathbb E \left[ (  \widehat{Y}_{t_{k+1}})^2 + (\widehat{Y}_{t_k} )^2 - 2 \widehat{Y}_{t_{k}} \widehat{Y}_{t_{k+1}} \right] \\
& \le 
&  \mathbb E \left[ (  \overline{Y}_{t_{k+1}})^2 + (\overline{Y}_{t_k} )^2 - 2 \mathbb E \left[ \overline Y_{t_{k+1}} \mathbb E\left[\overline Y_{t_{k}}| \widehat Y_{t_ {k}}\right] \big | \widehat Y_{t_{k+1}} \right] \right]\\
& = & {  \big| \big|   \overline{Y}_{t_{k+1}} - \overline{Y}_{t_k} \big| \big|}_{2}^2 \le \frac{\tilde c}{n},
\end{eqnarray*}
for a positive $\tilde c$ only depending on $\cL_3$ and where we recalled the $L^2$-estimate associated to the increments in the Euler scheme. 
To conclude it suffices to collect all the terms. 
\end{proof}

\section{Numerical tests}\label{sec:test}

In this section, we present two numerical experiments where we implement our quantization-based BSDE solver ad we test it. The first experiment involves a linear BSDE where the solution for the value process and the control is known in closed-form. This first test allows us to compare our newly proposed numerical approximation for the control with the closed-form solution. The second example focuses on a non-linear BSDE, with unknown closed-form solution, and  we compare the initial value of the solution according to our algorithm against a reference value available in the literature. The implementation of the routines was performed by means of the Java programming language and it is available at \url{https://github.com/AlessandroGnoatto}. Numerical tests were performed on a laptop equipped with a 4 core 2.9 GHz Intel Core i7 processor with 16 GB of RAM.

\subsection{A linear BSDE: hedging in the Black-Scholes model}

We first consider a linear FSDE of the form:
\begin{align*}
d Y_{t}=r Y_{t} d t+\sigma Y_{t} d W_{t}, \ Y_{0}=y_{0}>0,
\end{align*}
where $r=0.04$, $\sigma=0.25$ and $y_{0}=100$. We associate to this forward process the BSDE
\begin{align*}
U_{t}=\xi+\int_{t}^{T} f\left(s, Y_{s}, U_{s}, V_{s}\right) d s-\int_{t}^{T} V_{s}^\top d W_{s}, \quad t \in[0, T],
\end{align*}
with $$\xi = \left(Y_T-K\right)^+,\quad f\left(t, y, u, v\right) = -rv,$$ for $K=100$ and $T=1$. This corresponds to the well-known Black-Scholes model for the evaluation of a European Call option. For this BSDE the solution is analytically known, namely the process $Y$ is given by a pathwise application of the Black-Scholes formula, whereas the control satisfies
\begin{align*}
V_t=\frac{\partial U}{\partial y}(t,Y_t)\sigma Y_{t} = \mathcal  N \left( d_1(t,Y_t) \right) \sigma Y_{t},
\end{align*}
where $\mathcal N$ is the cumulative distribution function of the standard Gaussian and 
\begin{align*}
d_1(t,y):=\frac{\log \frac{y}{K}+\left(r+\frac{\sigma^2}{2}\right)(T-t)}{\sigma\sqrt{T-t}}.
\end{align*}
This example provides a validation of our proposed methodology in a simple case where a closed form solution to the BSDE is known. The exact solution for $U$, given the specified data, is $U_0=11.8370$. We apply our proposed algorithm by using a quantization grid consisting of $50$ points, a time discretization with $20$ points and a uniform mesh. The approximate initial value for the price $U$ is $11.7548$. Since the novelty of our approach is given by the new scheme for the control, we show that the scheme produces a reliable approximation for the control by comparing our approximation with the exact known solution. The reader is referred to Figure \ref{fig:hedging}, where we compare the exact and the quantization-based approximation for $V$ over the quantization grid. We observe that the newly proposed scheme provides a very good approximation.

\begin{figure}
\centering
  \subfloat{\label{fig:T20}\includegraphics[scale=0.45]{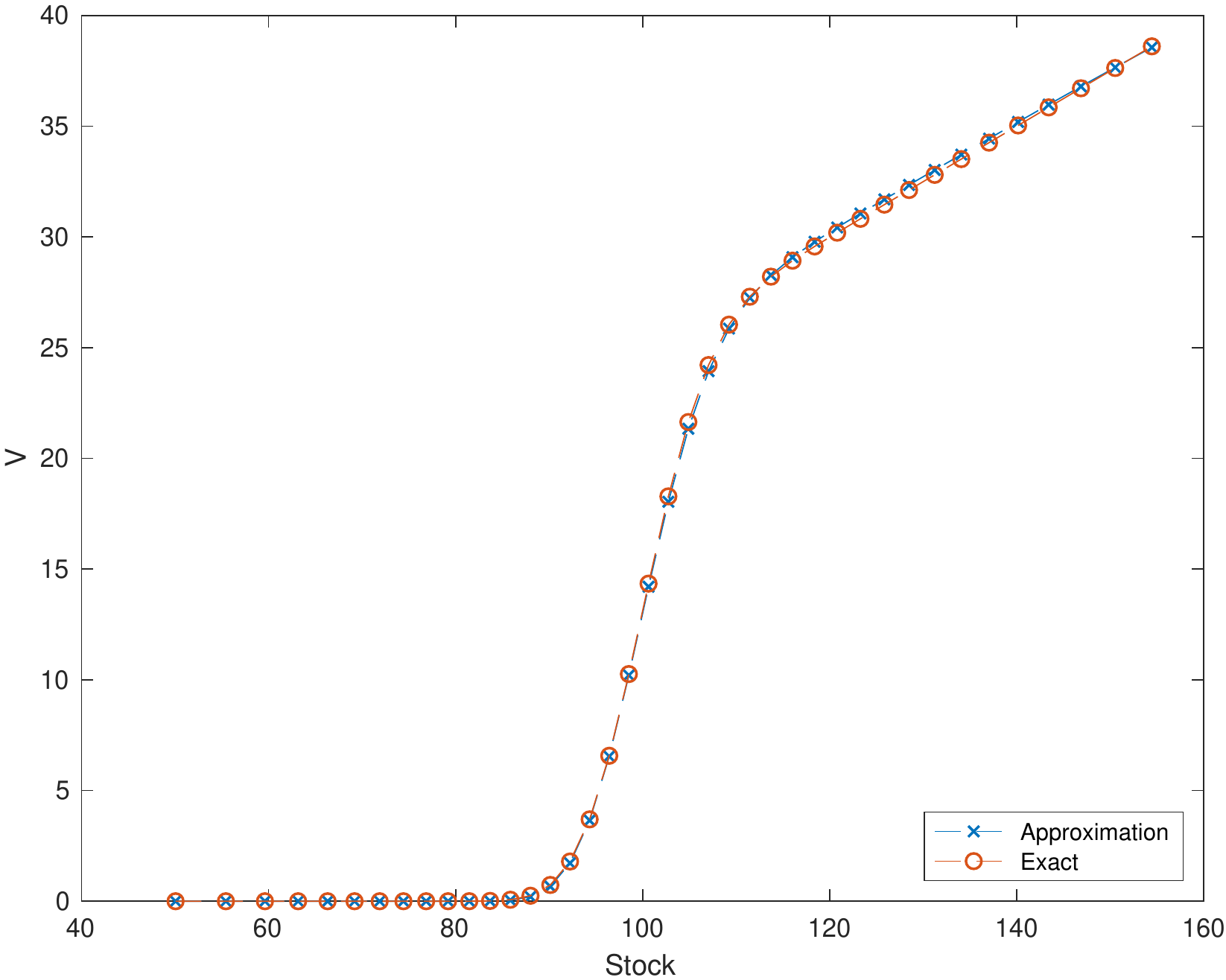}} \
    \subfloat{\label{fig:T15}\includegraphics[scale=0.45]{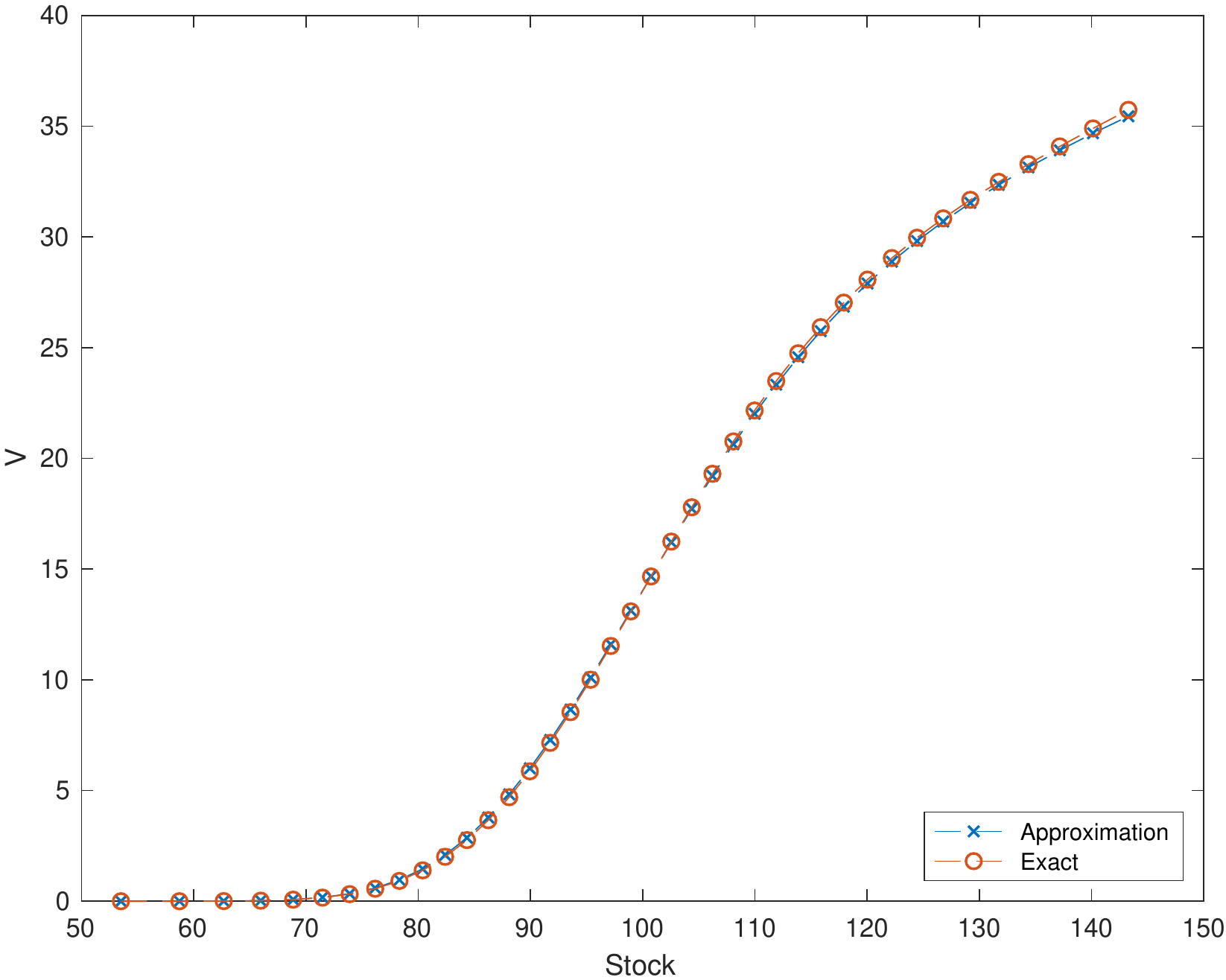}} \
     \subfloat{\label{fig:T10}\includegraphics[scale=0.45]{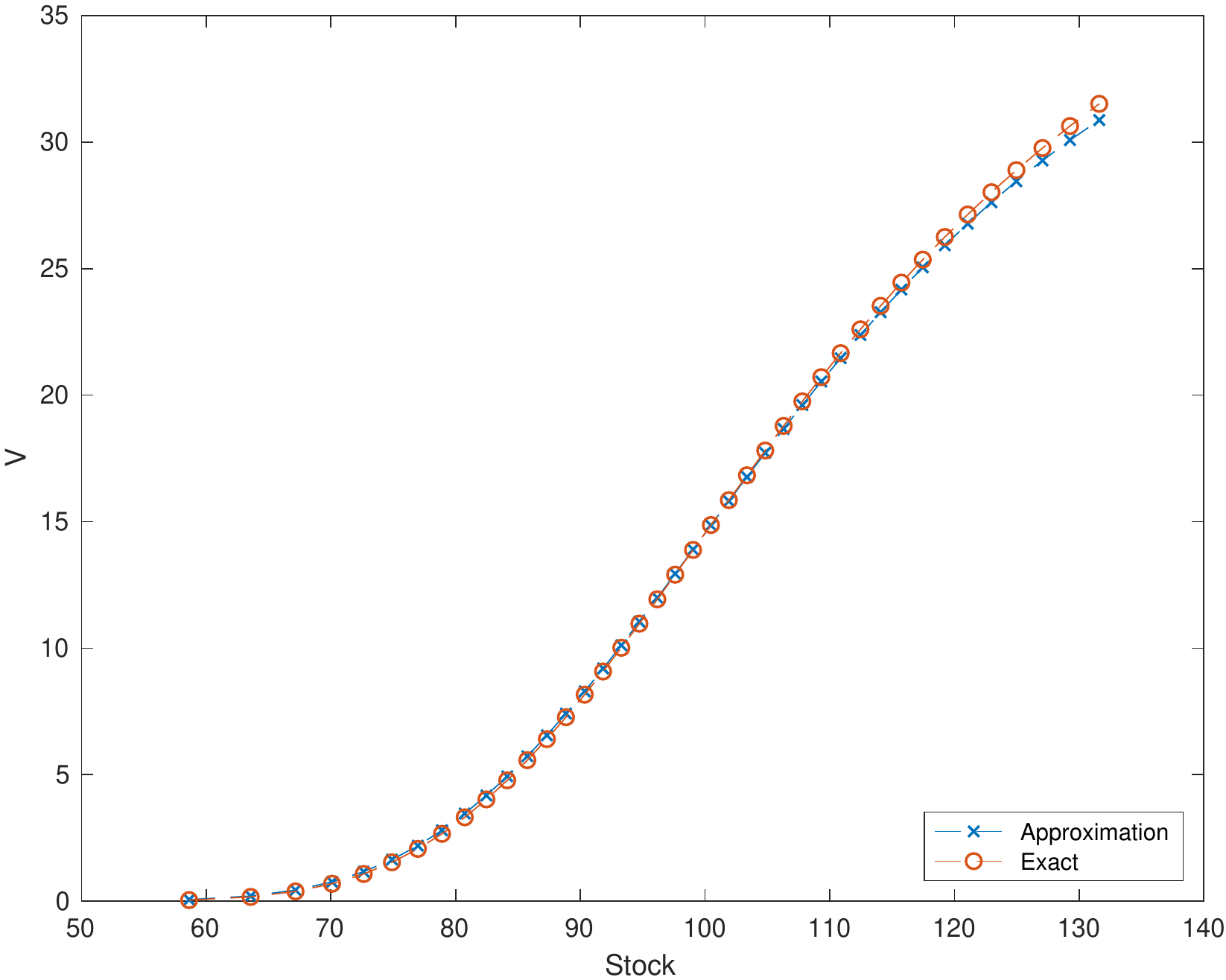}}\
      \subfloat{\label{fig:T5}\includegraphics[scale=0.45]{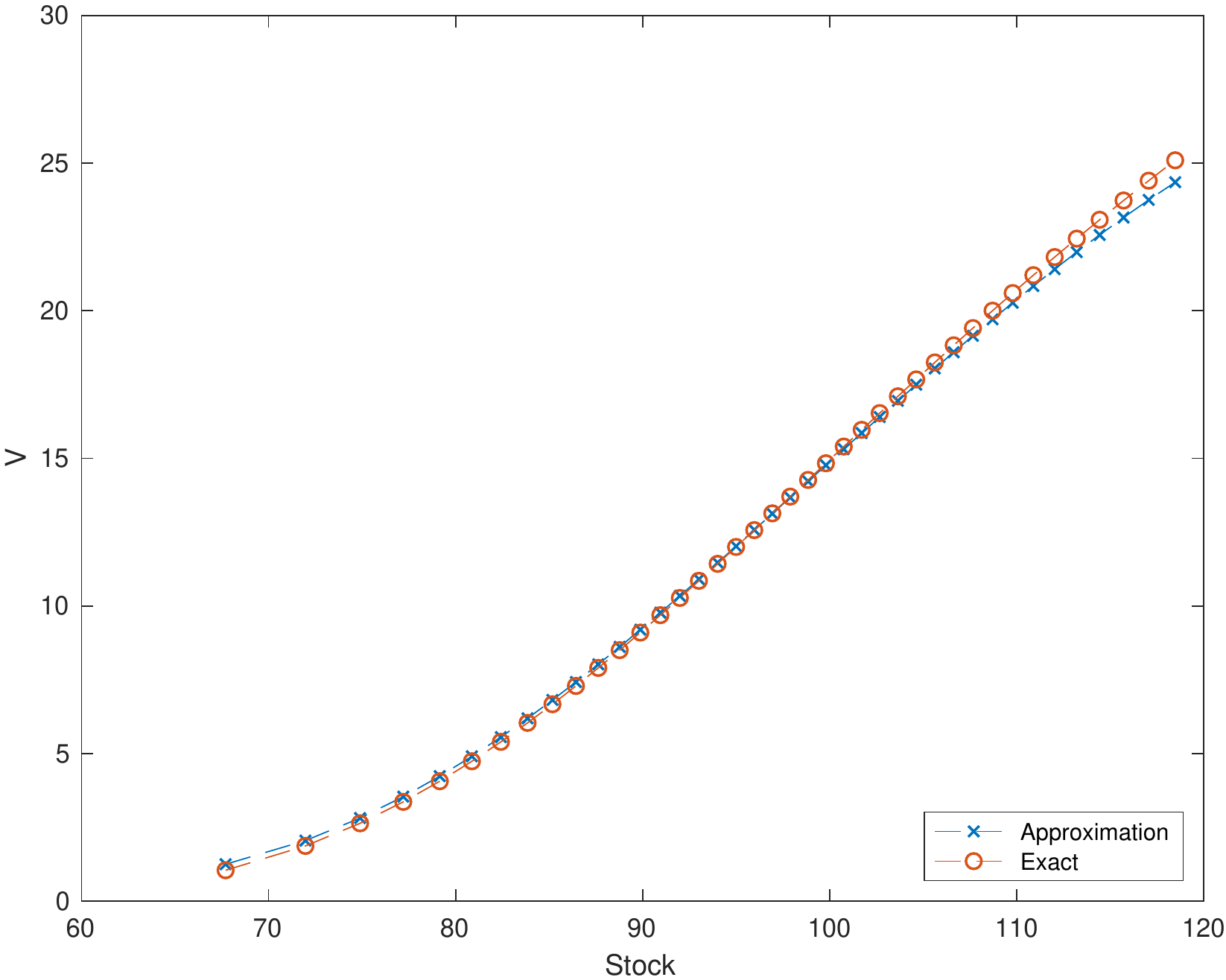}}  
\caption{Comparison between the exact and approximated hedging over the quantization grid at different time steps. Top left panel: $20th$ and terminal time step. Top right panel: $15th$ time step. Bottom left panel $10th$ time step. Bottom right panel $5th$ time step. \label{fig:hedging}}
\end{figure}

\subsection{A nonlinear BSDE: pricing with differential rates} We consider here a non-linear BSDE arising in financial mathematics in the context of option pricing, where we have the presence of two different interest rates, namely a borrowing and a lending rate, denoted respectively by $R$ and $r \le R$. Such setting corresponds to the model of \cite{Bergman95}. Our example is based on \cite{ps2018}.
We first consider a linear FSDE of the form:
\begin{align*}
d Y_{t}=\mu Y_{t} d t+\sigma Y_{t} d W_{t}, Y_{0}=y_{0}>0,
\end{align*}
where $\mu=0.05$, $\sigma=0.2$ and $y_{0}=100$. We associate to this forward process the BSDE
\begin{align*}
U_{t}=\xi+\int_{t}^{T} f\left(s, Y_{s}, U_{s}, V_{s}\right) d s-\int_{t}^{T} V_{s}^\top d W_{s}, \quad t \in[0, T],
\end{align*}
with
\begin{align*}
\xi &=\left(Y_{T}-K_{1}\right)^{+}-2\left(Y_{T}-K_{2}\right)^{+} ,\\
 f\left(t, y, u, v\right)& = -r u-\frac{\mu-r}{\sigma} v-(R-r) \min \left(u-\frac{v}{\sigma}, 0\right),
\end{align*}
where we set $K_1=95$, $K_2=105$ and $T=0.25$. This corresponds to a bull-Call spread with a long Call with strike $95$ and two short Call with strike $105$. We set the borrowing rate $R$ and the lending rate $r$ to: $R=0.06$ and $r=0.01$. There is no known analytical solution in this case. So, in line with what was done in Section 5.1 in \cite{ps2018} we benchmark our result against the reference value  $U_0=2.96$ (this was indeed taken by \cite[Section 4.2, Table 1]{BS2012}). We ran our algorithm by using $20$ quantization points and $50$ time discretization points and we obtained an estimate of the initial price of $2.9427$.  \textcolor{black}{However increasing the number of quantization points and time steps shows that the price is converging towards a lower value. For $100$ time steps and $100$ quantization points we obtain a value of $2.7782$, meaning that we can not confirm the value found in \cite{BS2012}.}

\begin{table}
\begin{tabular}{c|ccccc}
\hline
&5 &10 &20 &50 & 100\\
\hline
$5$&$2.8492$ & $3.1072$ & $3.4420$ & $3.9854$ & $4.4469$ \\
$10$&$2.9258$ & $2.9629$ & $3.0424$ & $3.2363$ & $3.5712$ \\
$15$&$2.8957$ & $2.8845$ & $2.9188$ & $3.0659$ & $3.2421$ \\
$20$&$2.8243$ & $2.8211$ & $2.8495$ & $2.9427$ & $3.0676$ \\
$50$&$2.8147$ & $2.7933$ & $2.7870$ & $2.7959$ & $2.8205$ \\
$100$&$2.8149$ & $2.7880$ & $2.7757$ & $2.7728$ & $2.7782$\\   
\hline
\end{tabular}
\caption{Option prices in the Bergman model. Each column corresponds to different numbers of time steps whereas each row corresponds to a different number of quantization points.\label{table:pricesBergman}}
\end{table}

\section{Conclusion}\label{sec:conclulsion}

We provided a useful modification  for the scheme of the control in \cite{ps2018} that allows to improve the algorithm for the approximation of the solution of a family of decoupled FBSDEs. Thanks to this simplification, we can apply a  fully based recursive marginal quantization approach that  does not involve  any  Monte Carlo simulation in any step of the procedure. 
We applied the scheme in some linear and non linear FBSDE examples and we found very good results even with a parsimonious number of quantization and time discretization points. This opens the door to more ambitious applications, like the computation of xVA on single and multiple positions, where   our fully quantization based method can be used as a pricing tool in the learning phase of any Neural Network based counterparty credit risk  algorithm, like the \textit{Deep xVA} approaches of \cite{gnoattoxva20} and \cite{albanesecrepey20} and \cite{crepeyxva20}.
\\

\bibliographystyle{apa}
\bibliography{biblio}
 
\end{document}